\documentclass[reqno,a4paper,11pt]{amsart}
\numberwithin{equation}{section}
\usepackage{lineno}
\usepackage{amssymb} 
\usepackage{amsmath} 
\usepackage{amsthm} 
\usepackage{ascmac} 
\usepackage{bm} 
\usepackage{color} 
\usepackage[OT2,T1]{fontenc} 
\usepackage[margin=1in]{geometry} 
\newcommand{\notmid}{\mathrel{\ooalign{$\mkern-5mu\not$\crcr$|$}}}
\usepackage{indentfirst} 
\usepackage{stmaryrd}
\usepackage{mathrsfs} 
\usepackage{multirow} 
\usepackage{graphicx} 
\usepackage{tikz-cd}
\usepackage{enumitem}
\usepackage{tabularray}
\usetikzlibrary{decorations.markings}

\usetikzlibrary{shapes.geometric}
\tikzset{every loop/.style={min distance=10mm,looseness=10}}
\usepackage{mathtools} 
\allowdisplaybreaks[4]
\mathtoolsset{showonlyrefs=true}

\theoremstyle{plain}
\newtheorem{thm}{Theorem} 
\newtheorem{prop}[thm]{Proposition}
\newtheorem{lem}[thm]{Lemma}
\newtheorem{cor}[thm]{Corollary}
\numberwithin{thm}{section} 
\theoremstyle{definition}
\newtheorem{defi}[thm]{Definition}
\newtheorem{conj}[thm]{Conjecture}
\newtheorem{ex}[thm]{Example}
\newtheorem{rem}[thm]{Remark}

\let\epsilon\relax
\makeatletter
\newcommand{\MyMathOperators}[1]{\@for\op:=#1\do{
	\expandafter\edef\csname\op\endcsname{\noexpand\mathop{\noexpand\mathrm{\op}}\nolimits}
}}
\makeatother
\MyMathOperators{%
	Ker,%
	Coker,%
	Hom,%
	End,%
	Aut,%
	id,%
	Gal,%
	ch,%
	ord,%
	Spec,%
	Specm,%
	Frac,%
	rank,%
	corank,%
	sgn,%
	Reg,%
	Sel,%
	Cl,%
	lcm,%
	gcd,%
	Ann,%
	alg,%
}

\renewcommand{\bar}[1]{\overline{#1}}

\newcommand{\blue}[1]{\textcolor{blue}{#1}} 

\DeclareSymbolFont{cyrletters}{OT2}{wncyr}{m}{n}
\DeclareMathSymbol{\Sha}{\mathalpha}{cyrletters}{"58}

\usepackage[OT2,T1]{fontenc}
\DeclareSymbolFont{cyrletters}{OT2}{wncyr}{m}{n}
\DeclareMathSymbol{\Sha}{\mathalpha}{cyrletters}{"58}
\newcommand{\Z}{\mathbb{Z}}
\newcommand{\Q}{\mathbb{Q}}

\newcommand{\F}{\mathbb{F}}

\usepackage[colorlinks=true,allcolors=black]{hyperref}

\title{Iwasawa Theory for K3 surfaces over finite fields}
\author{Rikuto Ito and Sohei Tateno}
\address{Rikuto Ito  \newline GRADUATE SCHOOL OF MATHEMATICS, NAGOYA UNIVERSITY, FURO-CHO, CHIKUSA-KU, NAGOYA, 464-8602, JAPAN}
\email{ito.rikuto.g7@s.mail.nagoya-u.ac.jp}
\address{Sohei Tateno \newline GRADUATE SCHOOL OF NATURAL SCIENCE \& TECHNOLOGY, KANAZAWA UNIVERSITY, KAKUMA-MACHI, KANAZAWA, 920-1192, JAPAN}
\email{inu.kaimashita@gmail.com}
\date{}
\begin{document}

\begin{abstract}
In this paper, we initiate  Iwasawa theory for K3 surfaces over finite fields. First, using the Artin--Tate conjecture, which is known to hold for  K3 surfaces, we prove an analogue of Mazur’s control theorem for elliptic curves over number fields. Second, we prove an analogue of Iwasawa’s class number formula for Brauer groups in two different ways. We also give explicit examples in the case of Kummer surfaces. Finally, we establish an analogue of the Iwasawa main conjecture for Brauer groups.
\end{abstract}
\maketitle

\tableofcontents
\section{Introduction}
\subsection{Iwasawa Theory for elliptic curves over number fields}
Let $p$ be a prime number. 
\begin{defi}
Let $M$ be a $\mathbb{Z}$-module. For each $n\geq 0$, let $M[p^{n}]\subset M$ be the submodule killed by $p^{n}$. The union $M[p^{\infty}]:=\bigcup_{n\geq 0} M[p^{n}]$ is called the \emph{$p$-primary part} of $M$. In particular, if $M$ is finite, then $M[p^{\infty}]$ coincides with the $p$-Sylow subgroup of $M$. 
\end{defi}
In 1959, Iwasawa proved the following celebrated formula, now known as Iwasawa's class number formula. \begin{thm}[Iwasawa \cite{iwasawa}]\label{oriiwasawa} Let $K$ be an arbitrary number field and let $K_{\infty}/K$ be a $\mathbb{Z}_{p}$-extension, i.e., $K_{\infty}/K$ is a Galois extension and $\Gamma:={\rm Gal}(K_{\infty}/K)\cong\mathbb{Z}_{p}$ as topological groups. For each integer $n\geq0$, let $K_{n}$ be  the intermediate field corresponding to the subgroup $p^{n}\mathbb{Z}_{p}$. Then there exist unique integers $\mu,  \lambda\in\mathbb{Z}_{\geq0}$ and $\nu\in\mathbb{Z}$ such that, for all sufficiently large $n\geq0$, we have 
\begin{equation}
    \#{\rm Cl}(K_{n})[p^{\infty}]=p^{\mu p^{n}+\lambda n+\nu},\end{equation}
    where ${\rm Cl}(K_{n})$ is the ideal class group of $K_{n}$. 
\end{thm}
This is regarded as the first formula describing the regularity of class numbers in certain towers of number fields. The integers $\mu$, $\lambda$, and $\nu$ are called \emph{the Iwasawa invariants of $K_\infty/K$}. The $\mu$ and $\lambda$-invariants are determined by the corresponding $\mu$ and $\lambda$-invariants of a finitely generated torsion $\Lambda$-module, where $\Lambda$ is the complete group ring $\mathbb{Z}_p[[\Gamma]]$. In the setting of Theorem \ref{oriiwasawa}, $\varprojlim_n {\rm Cl}(K_n)[p^\infty]$ is the corresponding $\Lambda$-module.

Analogues of Theorem \ref{oriiwasawa} are also known in the setting of elliptic curves over number fields. Let $E$ be an elliptic curve defined over a number field $K$.  Let $L/K$ be an arbitrary algebraic extension. Then it is known that there is an exact sequence
\begin{equation}\label{seltate}
    0\to E(L)\otimes \mathbb{Q}/\mathbb{Z}\to {\rm Sel}_{E}(L)\to \Sha_{E}(L)\to 0,
\end{equation}
where $E(L)$ denotes the group of $L$-rational points of $E$, ${\rm Sel}_{E}(L)$ is the classical Selmer group for $E$ over $L$, and $\Sha_{E}(L)$ is the Tate--Shafarevich group of $E$ over $L$.  In the Iwasawa theory for elliptic curves, a natural first object to consider is the cyclotomic $\mathbb{Z}_{p}$-extension $K_{\infty}$ which is defined as the unique subfield of $K(W)/K$ with $\Gamma:={\rm Gal}(K_{\infty}/K)\cong \mathbb{Z}_{p}$, where $W$ is the group of $p$-power roots of unity. Since $\Gamma$ acts on ${\rm Sel}_{E}(K_{\infty})[p^{\infty}]$ naturally, the Pontryagin dual ${\rm Sel}_{E}(K_{\infty})[p^{\infty}]^{\lor}:={\rm Hom}_{{\rm cont}}({\rm Sel}_{E}(K_{\infty})[p^{\infty}], \mathbb{Q}_{p}/\mathbb{Z}_{p})$ has a natural structure of a compact $\Lambda$-module. 
\begin{conj}[Mazur \cite{mazur}]\label{gentor}
    The $\Lambda$-module ${\rm Sel}_{E}(K_{\infty})[p^{\infty}]^{\lor}$ is a finitely generated torsion $\Lambda$-module.
\end{conj}
Some special cases of this conjecture  follow from the following Mazur's control theorem.
\begin{thm}[Mazur's control theorem \cite{mazur}]\label{mazurcont}Assume that $E$ has good ordinary reduction at all primes
of $K$ lying over $p$. Then the natural homomorphism
    \begin{equation}
        {\rm Sel}_{E}(K_{n})\to {\rm Sel}_{E}(K_{\infty})^{\Gamma^{n}}
    \end{equation}
    has bounded kernel and cokernel  as $n$ varies, where $\Gamma^{n}:={\rm Gal}(K_{\infty}/K_{n})$ and $(-)^{\Gamma^{n}}$ means the $\Gamma^{n}$-invariant part. 
\end{thm}
When $E$ is defined over $\mathbb{Q}$ and $K/\mathbb{Q}$ is a finite abelian extension, it is known that Conjecture \ref{gentor} holds for the base change $E\times_{\mathbb{Q}}{\rm Spec}(K)$ (Kato \cite{kato}, Rubin \cite{rubin}).

Assuming Conjecture  \ref{gentor}, Greenberg \cite{gre} proved that the Mordell--Weil ranks of $E(K_{n})$, for $n\geq0$, are bounded by the $\lambda$-invariant $\lambda_{E}$ of ${\rm Sel}_{E}(K_{\infty})[p^{\infty}]^{\lor}$. The Iwasawa-type formula for elliptic curves is the following.
\begin{thm}[Greenberg \text{\cite{gre}}]\label{iwasawatypeell} Let $E$ be an elliptic curve over a number field $K$. Assume that $E$ has good ordinary reduction at all primes
of $K$ lying over $p$. Suppose that Conjecture \ref{gentor} for $E/K$ is true and that the groups $\Sha_{E}(K_{n})[p^{\infty}]$ are finite for every $n\geq0$. Then there exist unique integers $\mu, \lambda\in \mathbb{Z}_{\geq0}$, and $\nu\in\mathbb{Z}$ such that, for sufficiently large $n\geq0$ we have
    \begin{equation}
        \#\Sha_{E}(K_{n})[p^{\infty}]=p^{\mu p^{n}+\lambda n+\nu}.
    \end{equation}
\end{thm}
Assuming Conjecture \ref{gentor}, let $\lambda^{M-W}$ be the maximum of the rank of $E(K_{n})$, and let $\mu_{E}$ be the $\mu$-invariant of ${\rm Sel}_{E}(K_{\infty})[p^{\infty}]^{\lor}$. Then $\mu$ and $\lambda$ in Theorem \ref{iwasawatypeell} are $\mu=\mu_{E}$ and $\lambda=\lambda_{E}-\lambda^{M-W}$.

By the short exact sequence \eqref{seltate}, the finiteness of ${\rm Sel}_{E}(K_{n})[p^{\infty}]$ is equivalent to the finiteness of $E(K_{n})$ and $\Sha_{E}(K_{n})[p^{\infty}]$. Hence, as a special case of Theorem \ref{iwasawatypeell}, we have
\begin{cor}\label{seliwasawa} Let $E$ be an elliptic curve over a number field $K$. Assume that $E$ has good ordinary reduction at all primes of $K$ lying over $p$. Suppose that the group ${\rm Sel}_{E}(K_{n})[p^{\infty}]$ is finite for every $n\geq0$. Then there exists a unique $\nu$ such that, for sufficiently large $n\geq0$, we have
    \begin{equation}
        \#{\rm Sel}_{E}(K_{n})[p^{\infty}]=p^{\mu_{E}p^{n}+\lambda_{E}n+\nu}.
    \end{equation}
\end{cor}
Each finitely generated torsion $\Lambda$-modules has an associated \emph{characteristic polynomial}, which is a generator of \emph{the characteristic ideal} (see Definition \ref{charideal}). The Iwasawa main conjecture claims equalities between characteristic ideals and $p$-adic L-functions. In the case of number fields, Mazur and Wiles \cite{mw} proved the conjecture for the cyclotomic $\Z_p$-extension $K_\infty/K$ of an arbitrary abelian extension $K/\Q$. {Wiles \cite{wiles1} later generalized the Iwasawa main conjecture to cyclotomic $\Z_p$-extensions of abelian extensions of totally real fields.

On the other hand, Mazur and Swinnerton-Dyer formulated an analogue  of the Iwasawa main conjecture for elliptic curves. 
\begin{conj}[Mazur and Swinnerton-Dyer \cite{ms}] \label{imcell}Choose a topological generator $\gamma\in \Gamma$, and fix an isomorphism $\Lambda\cong \mathbb{Z}_{p}[[T]]$ via $\gamma\mapsto 1+T$ (see Proposition \ref{idlambda}). Let $K/\mathbb{Q}$ be an abelian extension. Let $E_\mathbb{Q}$ be an elliptic curve and let $E$ denote the scalar extension to $K$. Assume that $K\cap \mathbb{Q}_{\infty}=\mathbb{Q}$ in a fixed algebraic closure $\overline{\mathbb{Q}}$. Suppose that $E$ has good ordinary reduction at all primes of $K$ lying over $p$. Then we have an equality of ideals in $\Lambda$
    \begin{equation}
        {\rm Char}_{\Lambda}({\rm Sel}_{E}(K_{\infty})[p^{\infty}]^{\lor})= (L_{p}(E, T)),
    \end{equation}
    where $L_{p}(E, T)$ is the $p$-adic L-function of $E/K$. 
\end{conj}
It is known that Conjecture \ref{imcell} is true in many cases (Kato \cite{kato}, Rubin \cite{rubin}, and Skinner--Urban \cite{su}). 
\subsection{Main results} Let $X$ be a K3 surface over a finite field $\mathbb{F}_{q}$, i.e., a smooth complete surface satisfying $ H^{1}(X, \mathcal{O}_{X})=0$ and $ \Omega^{2}_{X}\cong \mathcal{O}_{X}$. Let $p\notmid q$ be a prime. Fix an algebraic closure $\overline{\mathbb{F}}_{q}$. We consider a tower of finite fields
\begin{equation}
    \mathbb{F}_{q}\subset \mathbb{F}_{q^{p}}\subset\mathbb{F}_{q^{p^{2}}}\subset\cdots \subset \mathbb{F}_{q^{p^{n}}}\subset \cdots \subset\mathbb{F}_{q^{p^{\infty}}},
\end{equation}
where $\mathbb{F}_{q^{p^{\infty}}}:=\bigcup_{n\geq0} \mathbb{F}_{q^{p^{n}}}$. Since $\Gamma_{n}:={\rm Gal}(\mathbb{F}_{q^{p^{n}}}/\mathbb{F}_{q})\cong \mathbb{Z}/p^{n}\mathbb{Z}$, we have $\Gamma:={\rm Gal}(\mathbb{F}_{q^{p^{\infty}}}/\mathbb{F}_{q})\cong \mathbb{Z}_{p}$. Moreover,  we can consider the abelian extension $\mathbb{F}_{q^{p^{\infty}}}/\mathbb{F}_{q}$ as an analogue of cyclotomic $\mathbb{Z}_{p}$-extensions for number fields because the field $\mathbb{F}_{q^{p^{\infty}}}$ is contained in the field generated by all $p$-power roots of unity over $\mathbb{F}_{q}$. Let $X_{n}:=X\times_{\mathbb{F}_{q}}{\rm Spec}(\mathbb{F}_{q^{p^{n}}})$ and $X_{\infty}:=X\times_{\mathbb{F}_{q}}{\rm Spec}(\mathbb{F}_{q^{p^{\infty}}})$. Then we have a tower of finite \'etale morphisms 
\begin{equation}
  X\leftarrow  X_{1}\leftarrow X_{2}\leftarrow\cdots \leftarrow X_{n}\leftarrow \cdots \leftarrow X_{\infty},
\end{equation}
where the morphism $X\leftarrow X_{n}$ is a cyclic \'etale cover of the degree $p^{n}$, i.e, ${\rm Aut}(X_{n}/X)\cong \mathbb{Z}/p^{n}\mathbb{Z}$.  We call this tower  \emph{the constant $\mathbb{Z}_{p}$-tower over the K3 surface $X/\mathbb{F}_{q}$}.

\emph{The Artin\blue{--}Tate conjecture} (see Theorem \ref{artintate}) is the analogue, for algebraic surfaces over finite fields, of the BSD conjecture. In this analogy, the factor involving the order of the Tate--Shafarevich group in the BSD conjecture is replaced by the order of the \emph{Brauer group}. Since the Tate conjecture is known to hold for K3 surfaces over finite fields (Madapusi Pera \cite{mp15} and Kim--Madapusi Pera\cite{kmp}), the Artin–Tate conjecture is also known to hold for K3 surfaces. As an analogue of the Iwasawa-type formula (Theorem \ref{oriiwasawa} and Theorem \ref{iwasawatypeell}), we obtain the following formula.
\begin{thm}\label{iwasawatype}  Let $X/\mathbb{F}_{q}$ be a K3 surface. Let $X_{n}:=X\times_{\mathbb{F}_{q}}{\rm Spec}(\mathbb{F}_{q^{p^{n}}})$. Then there exist unique integers $\mu, \lambda\in\mathbb{Z}_{\geq0}$ and $\nu\in\Z$ such that, for every sufficiently large $n\geq 0$, we have
\[
\# {\rm Br}(X_n)[p^{\infty}]=p^{\mu p^n+\lambda n+\nu}.
\]
\end{thm}
The integers $\mu$, $\lambda$, and $\nu$ are called \emph{the Iwasawa invariants of the constant $\Z_p$-tower over $X$}, and we denote them by $\mu(X)$, $\lambda(X)$, and $\nu(X)$ respectively.

One way to prove this theorem is to use the following analogue, for K3 surfaces, of Mazur's control theorem (Theorem \ref{mazurcont}). 
\begin{thm}\label{cont0} Let $X/\mathbb{F}_{q}$ be a K3 surface. Let $X_{n}:=X\times_{\mathbb{F}_{q}}{\rm Spec}(\mathbb{F}_{q^{p^{n}}})$ and $X_{\infty}:=X\times_{\mathbb{F}_{q}}{\rm Spec}(\mathbb{F}_{q^{p^{\infty}}})$. Then the natural homomorphism
\begin{equation}
{\rm Br}(X_{n})[p^{\infty}]\to {\rm Br}(X_{\infty})^{\Gamma^{n}}[p^{\infty}]
\end{equation}
has bounded kernel and cokernel as $n$ varies, where $\Gamma^{n}:={\rm Gal}(\mathbb{F}_{q^{p^{\infty}}}/\mathbb{F}_{q^{p^{n}}})$.
\end{thm}
\begin{rem}
    When $p\mid q$, Lai, Longhi, Suzuki, Tan, and Trihan \cite[Remark 3.1.2]{suzuki} have proved a control theorem of the Brauer groups for certain classes of surfaces over $\mathbb{F}_{q}$, including elliptic K3 surfaces.
\end{rem}
Using the proof of Theorem \ref{cont0}, we show that the following $\Lambda$-module
\begin{equation}\label{arimod}
    {\rm Br}(X_{\infty})[p^{\infty}]^{\lor}
\end{equation}
is a finitely generated torsion $\Lambda$-module (Corollary \ref{tormod}). The invariants $\mu(X)$, $\lambda(X)$ are determined by the Iwasawa invariants $\mu({\rm Br}(X_{\infty})[p^{\infty}]^{\lor})$, $\lambda({\rm Br}(X_{\infty})[p^{\infty}]^{\lor})$ of the $\Lambda$-module \eqref{arimod}. 

    Another proof of Theorem \ref{iwasawatype} uses the transcendental part $L_{{\rm tr}}(X, t)\in \mathbb{Z}_{p}[t]$ of the L-function $L(X, t)\in\mathbb{Z}_{p}[t]$ (see Definition \ref{Ltr} and Definition \ref{lfunction}). Via the ring homomorphism $\mathbb{Z}_{p}[t]\ni t\mapsto 1+T\in \mathbb{Z}_{p}[[T]]$, the polynomial  $L_{{\rm tr}}(X, t)$ can be regarded as an element $L_{{\rm tr}}(X, 1+T)$ of $\mathbb{Z}_{p}[[T]]$. Then we can define the $\mu$-invariant $\mu(L_{\rm tr}(X, 1+T))$ and the  $\lambda$-invariant $\lambda(L_{\rm tr}(X, 1+T))$ of $L_{{\rm tr}}(X, 1+T)$ (Definition \ref{polyiwa}). Then Theorem \ref{iwasawatype} follows from the Artin-Tate conjecture and the fundamental evaluation formula (Theorem \ref{fef}), thereby showing $\mu(X)=\mu(L_{\rm tr}(X, 1+T))$ and $\lambda(X)=\lambda(L_{\rm tr}(X, 1+T))$. By computing the $\mu(L_{\rm tr}(X, 1+T))$ directly, we obtain the following theorem (Corollary \ref{coinv}). 
\begin{thm}\label{mu0}
Let $X/{\F}_q$ be a K3 surface. Then $\mu(X)$ is equal to $0$. 
\end{thm}
    This is an analogue of the vanishing of the $\mu$-invariant in the Iwasawa theory for function fields of one variable over finite fields (Washington \cite[Section 7.4]{Washington}). From the analogue of function fields, Iwasawa \cite{iwasawa2} conjectured that for every cyclotomic $\mathbb{Z}_{p}$-extension of number fields, the $\mu$-invariant is $0$. When the number field is an abelian extension over $\mathbb{Q}$, Ferrero and Washington \cite{fer} proved that the $\mu$-invariant is $0$. 
    
    Therefore, our two different proofs of Theorem \ref{iwasawatype} yield the equalities
    \begin{equation}
        \mu({\rm Br}(X_{\infty})[p^{\infty}]^{\lor})=\mu(L_{{\rm tr}}(X, 1+T))=0, ~\lambda({\rm Br}(X_{\infty})[p^{\infty}]^{\lor})=\lambda(L_{{\rm tr}}(X, 1+T)).
    \end{equation}
    More strongly, we obtain an analogue of the Iwasawa main conjecture for elliptic curves (Conjecture \ref{imcell}).
\begin{thm}\label{Iwasawa main conj k3}  Let $\gamma\in\Gamma$ be the topological generator induced from the arithmetic Frobenius $\overline{\mathbb{F}}_{q}\ni x\mapsto x^{q}\in \overline{\mathbb{F}}_{q}$. Fix the isomorphism $\Lambda\cong \mathbb{Z}_{p}[[T]]$ via $\gamma\mapsto 1+T$. Let $X/\mathbb{F}_{q}$ be a K3 surface, and let $X_{\infty}:=X\times_{\mathbb{F}_{q}}{\rm Spec}(\mathbb{F}_{q^{p^{\infty}}})$. Let $p\notmid q$ be a prime. Then we have an equality of ideals in $\Lambda$
\begin{equation}
{\rm Char}_{\Lambda}({\rm Br}(X_{\infty})[p^{\infty}]^{\lor})=(L_{{\rm tr}}(X, 1+T)).
\end{equation} 
\end{thm}
From our main results, we obtain the following analogy between elliptic curves and K3 surfaces.
\begin{center}
\begin{tblr}{|l|r|r|r|} \hline
    \SetCell[c=1]{c}object  & \SetCell[c=1]{c}  Iwasawa module & \SetCell[c=1]{c}Iwasawa-type formula &\SetCell[c=1]{c}L-function\\ \hline
 elliptic curve $E/K$& ${\rm Sel}_{E}(K_{\infty})[p^{\infty}]^{\lor}$ &$\#\Sha_{E}(K_{n})[p^{\infty}]=p^{\mu p^{n}+\lambda n+\nu}$ & $L_{p}(E, T)$ \\
  K3 surface $X/\mathbb{F}_{q}$ &${\rm Br}(X_{\infty})[p^{\infty}]^{\lor}$  &$\# {\rm Br}(X_n)[p^{\infty}]=p^{\mu p^n+\lambda n+\nu}$ & $L_{{\rm tr}}(X, 1+T)$ \\ \hline
 \end{tblr}
\end{center}
\begin{rem}
    We have seen that the Iwasawa module ${\rm Br}(X_{\infty})[p^{\infty}]^{\lor}$ is the projective limit of ${\rm Br}(X_{n})[p^{\infty}]^{\lor}$ (Proposition \ref{nlim}), and we know that each ${\rm Br}(X_{n})[p^{\infty}]$ is finite. From this viewpoint, if we assume that ${\rm Sel}_{E}(K_{n})[p^{\infty}]$ is finite for every $n$, then Corollary \ref{seliwasawa} shows that, in the above table, the term corresponding to $\Sha_{E}(K_{n})[p^{\infty}]$ should be replaced by ${\rm Sel}_{E}(K_{n})[p^{\infty}]$.
\end{rem}
\subsection{The structure of this paper}
Section \ref{iwasawatheory} contains  preliminaries for studying  the Iwasawa module ${\rm Br}(X_{\infty})[p^{\infty}]^{\lor}$.

The core of our main theorems is the proof of the control theorem (especially, Theorem \ref{geoarith}). To prove it, we introduce another Iwasawa module ${\rm Br}(\overline{X})^{G^{\infty}}[p^{\infty}]^{\lor}$ in Section \ref{briwasawa}. This module may be regarded as an intermediate object between ${\rm Br}(X_{\infty})[p^{\infty}]^{\lor}$ and $L_{{\rm tr}}(X, 1+T)$. We  show that the Iwasawa module ${\rm Br}(\overline{X})^{G^{\infty}}[p^{\infty}]^{\lor}$ is a finitely generated torsion $\Lambda$-module  by using the general theory developed in Section \ref{iwasawatheory} and a consequence of the Tate conjecture for K3 surfaces (Colliot-Th\`el\'ene and Skorobogatov \cite[16.1.4]{cs}).

 We prove the control theorem (Theorem \ref{K3conto}) in Section \ref{control theorem}. The most important tool to prove the control theorem is the Artin-Tate conjecture for K3 surfaces (Theorem \ref{artintate}). Thanks to this theorem, we can compare the orders of ${\rm Br}(X_{\infty})[p^{\infty}]^{\lor}$ and ${\rm Br}(\overline{X})^{G^{\infty}}[p^{\infty}]^{\lor}$ in terms of the polynomial $L_{{\rm tr}}(X, 1+T)$.

In Section \ref{iwaproof}, we give two independent proofs of the Iwasawa-type formula (Theorem \ref{iwasawatype}): one is derived from the Iwasawa invariants of ${\rm Br}(\overline{X})^{G^{\infty}}[p^{\infty}]^{\lor}$, and the other is obtained by using the Iwasawa invariants of  $L_{{\rm tr}}(X, 1+T)$ based on the fundamental evaluation formula (Theorem \ref{fef}). As an example, we derive a class number formula for Kummer surfaces associated with products of non-isogenous elliptic curves. We then give explicit computations of the Iwasawa invariants for specific elliptic curves.

In the final section, by computing the characteristic ideal of the $\Lambda$-module ${\rm Br}(\overline{X})^{G^{\infty}}[p^{\infty}]^{\lor}$, we prove  the Iwasawa main conjecture (Theorem \ref{imck3}).
 
We expect that the Iwasawa theory for K3 surfaces developed here will serve as a first step toward analogous theories over global function fields and number fields, as well as toward an Iwasawa theory for a broader class of algebraic surfaces over finite fields.

\textbf{Acknowledgements.} The first author would like to express his deepest gratitude to his supervisor, Sho Tanimoto. It is thanks to his continuous support that the author has been able to continue pursuing mathematics.

The authors are also grateful to Alexei Skorobogatov for his various helpful comments and information, including pointing out an inaccuracy in the citation for Theorem \ref{ls}, as well as for his encouragement.

The authors would also like to thank Taiga Adachi, Bryden Cais, Thomas Geisser, Takashi Hara, Takenori Kataoka, Iwao Kimura, Yuki Matsuoka, Kosuke Mizuno, Ryosuke Murooka, Asuka Shiga, Takashi Suzuki, and ChatGPT for their informative comments and encouragements.

The first author was supported by  JST FOREST program Grant number JPMJFR212Z. The second author was supported by JSPS KAKENHI, Grant Number 26KJ0138.
\section{Iwasawa Theory}\label{iwasawatheory}
In this section, we review the basic notions of Iwasawa theory.
\subsection{Fundamental Evaluation Formula in Iwasawa Theory}
In this subsection, we briefly summarize $p$-adic evaluation theory in Iwasawa theory for the reader's convenience. Our basic reference is Washington \cite[Section 7]{Washington}. In particular, we give a convenient formulation of a useful classical result, which we call \emph{the fundamental evaluation formula in Iwasawa theory}. Although this result is implicit in many standard references, explicit statements in the form needed here seem to be less common.

Let $p$ be a prime number. Let $K/\Q_p$ be a finite extension, and let $\mathcal{O}$ be its valuation ring. Let $\mathfrak{p}$ denote its maximal ideal, and let $\pi$ denote its uniformizer, i.e., a generator of the ideal $\mathfrak{p}$.  Let $\mathcal{O}\llbracket T\rrbracket$ be the power series ring with coefficients in $\mathcal{O}$. 

We fix an embedding $\bar{\Q}\hookrightarrow \bar{\Q}_p$. Let $v_p: \bar{\Q}_p\to \Q$ be the $p$-adic valuation normalized so that $v_p(p)=1$. Hence $v_p(\pi)=\frac{1}{e}$, where $e$ is the ramification index of $p$ in $\mathcal{O}$.

\begin{defi}[Washington \text{\cite[Section 7]{Washington}}]\label{polyiwa}
A polynomial $f(T)\in \mathcal{O}[T]$ is said to be \emph{distinguished} if $f(T)$ is monic and all non-leading coefficients lie in $\mathfrak{p}$.

Let $f(T)\in\mathcal{O}\llbracket T\rrbracket$. By the $p$-adic Weierstrass preparation theorem, there  exist  a unique non-negative integer $N\geq 0$, a unique distinguished polynomial $g(T)\in \mathcal{O}[T]$, and a unique unit $u(T)\in\mathcal{O}\llbracket T\rrbracket^*$ such that $f(T)=\pi^{N}g(T)u(T)$. Let $e$ be the ramification index of $K/\Q_p$. $\mu=\mu(f):=\frac{N}{e}$ is called \emph{the Iwasawa $\mu$-invariant of $f$}. $\lambda=\lambda(f):=\deg g$ is called \emph{the Iwasawa $\lambda$-invariant of $f$}.
\end{defi}
\begin{rem}
    The numbers $\mu(f)$ and $\lambda(f)$ are independent of the choice of the coefficient field $K/\mathbb{Q}_{p}$. In fact, let $K^{'}/\mathbb{Q}_{p}$ be another field and let $\mathcal{O}^{'}$ be the ring of integers. Suppose that $f\in \mathcal{O}^{'}[[T]]$. We may assume that $\mathcal{O}\subset\mathcal{O}^{'}$. Let $\pi^{'}$ be the uniformizer of the prime ideal $\mathfrak{p}^{'}$ of $\mathcal{O}^{'}$. Since there exists an integer $m\geq0$ such that $\mathfrak{p}=\mathfrak{p}^{'m}$, the ramification index of $p$ in $\mathcal{O}^{'}$ is $me$ and we have $f(T)=\pi^{'Nm}g(T)u^{'}(T)$ where $u^{'}(T)$ is a product of a unit in $\mathcal{O}^{'}$ and $u(T)$. In particular,  $\mu, \lambda$ are independent of the choice of the field $K^{'}$. 
\end{rem}
For each $n\geq 0$, let $W_n$ be the set of $p^n$-th roots of unity, and let $W=\bigcup_{n\geq 0} W_n$. Let $\varphi$ denote Euler's totient function.

\begin{thm}[Fundamental evaluation formula in Iwasawa theory]\label{fef}
Let $f(T)\in\mathcal{O}\llbracket T\rrbracket$.  Let $n_0$ be the smallest non-negative integer such that $e\cdot \lambda(f)<\varphi(p^{n_0})$. Assume that $f(\zeta-1)\neq 0$ for all $\zeta \in(W\setminus\{1\})$. Then there exists a unique $\nu\in\Q$ such that, for every $n\geq n_0$, we have $v_p(\prod_{\zeta\in(W_n\setminus\{1\})}f(\zeta-1))=\mu(f) p^n+\lambda(f)n+\nu$.
\end{thm}
\begin{proof}
By the Weierstrass preparation theorem, there exist a unique non-negative integer $N\geq 0$, a distinguished polynomial $g(T)\in \mathcal{O}[T]$, and a unit $u(T)\in\mathcal{O}\llbracket T\rrbracket^*$ such that $f(T)=\pi^{N}g(T)u(T)$. 

Since $u(T)$ is a unit, $u(T)$ can be written in the form $u(T)=a_0+Tq(T)$, where $a_0\in \mathcal{O}^*$ and $q(T)\in \mathcal{O}\llbracket T\rrbracket$. Since $v_p(a_0)=0$ and $v_p(\zeta-1)>0$, the non-archimedean property implies that $v_p(u(\zeta-1))=0$.

Now we evaluate the $g(T)$-part. Since $g(T)$ is distinguished, $g(T)$ can be written in the form
\[
g(T)=T^{\lambda}+a_{\lambda-1}T^{\lambda-1}+\cdots+a_1T+a_0
\]
satisfying $v_p(a_i)\geq \frac{1}{e}$. Let $\zeta$ be a primitive $p^n$-th root of unity. By \cite[III-Lemma 3]{cassels}, we obtain
\[
v_p(\zeta-1)=\frac{1}{\varphi(p^n)}=\frac{1}{p^{n-1}(p-1)},
\]
thus we have $v_p((\zeta-1)^{\lambda})=\frac{\lambda}{\varphi(p^n)}$. Hence, for every $n\geq n_0$ we have $v_p((\zeta-1)^{\lambda})=\frac{\lambda}{\varphi(p^n)}<\frac{1}{e}$. Thus, for every $0\leq i<\lambda$, we have
\begin{align*}
v_p(a_i(\zeta-1)^i)&=  v_p(a_i)+v_p((\zeta-1)^i) \geq   \frac{1}{e}+\frac{i}{\varphi(p^n)} \geq \frac{1}{e} >\frac{\lambda}{\varphi(p^n)}=v_p((\zeta-1)^{\lambda}).
\end{align*}
Since the valuation $v_p$ is non-archimedean, we have
\[
v_p(g(\zeta-1))=v_p((\zeta-1)^{\lambda}+a_{\lambda -1}(\zeta -1)^{\lambda -1}+\cdots+a_1(\zeta-1)+a_0)=v_p((\zeta-1)^{\lambda})=\frac{\lambda}{\varphi(p^n)}.
\]
Hence $v_p(\prod_{\substack{\zeta^{p^n}=1\\ \zeta^{p^{n-1}}\neq 1}} g(\zeta-1))=\varphi(p^n)\cdot\frac{\lambda}{\varphi(p^n)}=\lambda$.
Thus we obtain
\begin{align*}
v_p(\prod_{\zeta\in (W_n\setminus\{1\})} g(\zeta-1))&= \sum_{k=1}^n v_p(\prod_{\substack{\zeta^{p^k}=1 \\ \zeta^{p^{k-1}}\neq 1}} g(\zeta -1))\\&= \sum_{k=1}^{n_0 -1} v_p(\prod_ {\substack{\zeta^{p^k}=1 \\ \zeta^{p^{k-1}}\neq 1}} g(\zeta -1))+\sum_{k=n_0}^n v_p(\prod_{\substack{\zeta^{p^k}=1 \\ \zeta^{p^{k-1}}\neq 1}} g(\zeta-1))
=\lambda n+\nu_g,
\end{align*}
where $\nu_g= \sum_{k=1}^{n_0 -1} v_p(\prod_ {\substack{\zeta^{p^k}=1 \\ \zeta^{p^{k-1}}\neq 1}} g(\zeta -1))-\lambda(n_0-1)$. 

Consequently, we obtain
\begin{align*}
v_p(\prod_{\zeta\in (W_n\setminus\{1\})}f(\zeta-1))&=\sum_{\zeta\in (W_n\setminus\{1\})} v_p(\pi^N)+\sum_{\zeta\in (W_n\setminus\{1\})} v_p(g(\zeta-1))+\sum_{\zeta\in (W_n\setminus\{1\})} v_p(u(\zeta -1))\\ &=\mu (p^n-1)+\lambda n +\nu_g=\mu p^n+\lambda n +\nu,
\end{align*}
where $\nu=\nu_g-\mu$. This completes the proof.
\end{proof}
\subsection{Iwasawa Modules} In this subsection, we discuss module-theoretic aspects of Iwasawa theory. We define some module-theoretic invariants. First, we introduce the structure theorem of finitely generated torsion modules  over an integrally closed noetherian domain. 
\begin{thm}[General Structure Theorem]\label{genstr} Let $R$ be an integrally closed noetherian domain. 
    Let $M$ be a finitely generated torsion $R$-module. Then there exist unique prime ideals $\mathfrak{p}_{1}, \cdots, \mathfrak{p}_{u}$ of $R$ of height $1$ and unique non-negative integers $q_{1}, \cdots, q_{u}$, up to permutation, such that the $R$-module $M$ is pseudo-isomorphic to $\bigoplus^{u}_{i=1}R/\mathfrak{p_{i}}^{q_{i}}$. 
\end{thm}
\begin{proof}
    For example, see \cite[Theorem 2.36]{o1}.
\end{proof}
For an integrally closed noetherian domain $R$, let $P^{1}(R)$ be the set of prime ideals of $R$ of height 1. Since $R$ is integrally closed, for every $\mathfrak{p}\in P^{1}(R)$, the localization $R_{\mathfrak{p}}$ is a DVR. For $\mathfrak{p}\in P^{1}(R)$ and a finitely generated torsion $R$-module $M$, let $l(\mathfrak{p}, M)$ be the length of the $R_{\mathfrak{p}}$-module $M_{\mathfrak{p}}$. 
\begin{defi}[Characteristic ideals (\text{Ochiai \cite[Definition 2.38]{o1}})]\label{charideal}
    Let $R$, $M$, and $\mathfrak{p}_{i}$ be as in Theorem \ref{genstr}. Then the \emph{characteristic ideal} of $M$ is an ideal in $R$ defined as 
    \begin{equation}
        {\rm Char}_{R}(M):=\{x\in R~|~{\rm ord}_{\mathfrak{p}}(x)\geq l(\mathfrak{p}, M), \forall \mathfrak{p}\in P^{1}(R)\},
    \end{equation}
    where ${\rm ord}_{\mathfrak{p}}(-)$ is the normalized discrete valuation of $R_{\mathfrak{p}}$. 
\end{defi}
When $R=\mathbb{Z}_{p}$, by the structure theorem for finitely generated modules over a PID, the finitely generated torsion $\mathbb{Z}_{p}$-module is of the form $\mathbb{Z}_{p}/p^{m}\mathbb{Z}_{p}$ for a unique  $m\geq0$. Hence we have
\begin{equation}\label{simpchar}
    {\rm Char}_{\mathbb{Z}_{p}}M=(p^{m}).
\end{equation}

Now, let $q$ be a power of a prime number, and assume that  $p$ does not divide $q$. Let $\Gamma:={\rm Gal}(\mathbb{F}_{q^{p^{\infty}}}/\mathbb{F}_{q})$ be the Galois group.  The \emph{Iwasawa algebra} is the following complete group ring (topological inverse limit of group rings):
\begin{align}
\Lambda:=\mathbb{Z}_{p}[[\Gamma]]:=\varprojlim_{n}\mathbb{Z}_{p}[\Gamma_{n}],
\end{align}
where $\Gamma_{n}:={\rm Gal}(\mathbb{F}_{q^{p^{n}}}/\mathbb{F}_{q})$. It is known that the Iwasawa algebra is a compact $\mathbb{Z}_{p}$-module. 

On the other hand, let  $\mathbb{Z}_{p}[[T]]$ be the power series ring in one variable over $\mathbb{Z}_{p}$. By \cite[Proposition 13.9]{Washington} and \cite[Lemma 13.11]{Washington}, the ring $\mathbb{Z}_{p}[[T]]$ is a noetherian local ring with the maximal ideal $(p, T)$. Moreover, the ring $\mathbb{Z}_{p}[[T]]$ is complete with respect to the $(p, T)$-adic topology (see \cite[23.J]{matsumura}).  
\begin{prop}\label{idlambda} Let $\gamma\in \Gamma$ be a topological generator. Then  the map 
\begin{align} \Lambda&\to \mathbb{Z}_{p}[[T]],
\\  \gamma&\mapsto  1+T,\end{align}
is an isomorphism of topological $\mathbb{Z}_{p}$-algebras. 
\end{prop}
\begin{proof}
    See \cite[Proposition 5.3.5]{nsw}.
\end{proof}
\begin{defi} A compact $\Lambda$-module is called an \emph{Iwasawa module}.  
\end{defi}
The following lemma will be used to construct the Iwasawa modules ${\rm Br}(X_{\infty})[p^{\infty}]^{\lor}$ and ${\rm Br}(\overline{X})^{G^{\infty}}[p^{\infty}]^{\lor}$ in Section \ref{briwasawa}.
\begin{lem}\label{natiwasawa}
 If   a compact $\mathbb{Z}_{p}$-module $M$ has a continuous action of $\Gamma$, then $M$ has a natural structure of an Iwasawa module. 
\end{lem}
\begin{proof}
    See \cite[Lemma 2.23]{o1}.
\end{proof}

\begin{thm}[Structure theorem of Iwasawa modules]\label{iwstr} If $M$ is a finitely generated torsion $\mathbb{Z}_{p}[[T]]$-module, then there exists an exact sequence of $\mathbb{Z}_{p}[[T]]$-modules:
    \begin{equation}
        0\to Z\to M\to \bigoplus^{s}_{i=1}\mathbb{Z}_{p}[[T]]/(p^{m_{i}})\oplus\bigoplus^{t}_{j=1}\mathbb{Z}_{p}[[T]]/(F_{j}(T)^{n_{j}})\to Z^{'}\to 0,
    \end{equation}
    where $s, t$ are non-negative integers, $Z, Z^{'}$ are finite abelian groups, and $F_{j}(T)$ are non-unit irreducible distinguished polynomials. 
\end{thm}
\begin{proof}
    See \cite[Theorem 2.39]{o1}.
\end{proof}
By Theorem \ref{genstr}, $s, t, m_{i}, n_{j}$, and prime ideals $(F_{j}(T))$ are unique, up to permutation. Put ${\rm char}(M):= \prod^{s}_{i=1}p^{m_{i}}\prod^{t}_{j=1}F_{j}(T)^{n_{j}}. $
Then we have  
\begin{equation}\label{chardef}
    {\rm Char}_{\mathbb{Z}_{p}[[T]]}(M)=\left({\rm char}(M)\right).
\end{equation}
The polynomial ${\rm char}(M)$ is called \emph{the characteristic polynomial} of $M$.
\begin{defi}(Iwasawa invariants)\label{moduinv} Let $M$, $s$, $t$, $m_{i}$, $n_{j}$, and $(F_{j}(T))$ be as in Theorem \ref{iwstr}. For the $\mathbb{Z}_{p}[[T]]$-module $M$, we define the \emph{$\mu$-invariant} $\mu(M)$ and \emph{$\lambda$-invariant $\lambda(M)$} by  the non-negative integers $\mu({\rm char}(M))$ and  $\lambda({\rm char}(M))$ respectively. 
    \end{defi}
    From the expression of the characteristic ideal \eqref{chardef}, we obtain $\mu(M)=\sum^{s}_{i=1}m_{i}$, $\lambda(M)=\sum^{t}_{j=1}{\rm deg}(F_{j}(T))n_{j}$.
\begin{thm}\label{moduiwasawa} Fix a topological generator $\gamma\in \Gamma$. 
    Let $M$ be a finitely generated torsion $\Lambda$-module. Suppose that for every integer $n\geq0$, we have $\# (M/(\gamma^{p^{n}}-1)M)<\infty$. Then there exists a unique $\nu\in\mathbb{Z}$ such that,  for every sufficiently large $n\geq 0$, we have 
    \begin{equation}
        \# (M/(\gamma^{p^{n}}-1)M)=p^{\mu(M)p^n+\lambda(M)n+\nu}.
    \end{equation}
\end{thm}
\begin{proof}
    See \cite[Theorem 2.45]{o1}.
\end{proof}
\section{The Iwasawa Modules ${\rm Br}(X_{\infty})[p^{\infty}]^{\lor}$ and ${\rm Br}(\overline{X})^{G^{\infty}}[p^{\infty}]^{\lor}$}\label{briwasawa}
In this section, we will describe some properties of certain Brauer groups. We introduce two $\mathbb{Z}_{p}$-modules ${\rm Br}(X_{\infty})[p^{\infty}]^{\lor}$ and ${\rm Br}(\overline{X})^{G^{\infty}}[p^{\infty}]^{\lor}$.
\subsection{The direct system $({\rm Br}(X_{n})[p^{\infty}]\to {\rm Br}(X_{m})[p^{\infty}])_{n\geq m}$}
Let $q$ be a power of a prime number. Let $X/\mathbb{F}_{q}$ denote a  K3 surface defined over $\mathbb{F}_{q}$, and let $p$ be a prime number with $p\notmid q$. Fix an algebraic closure $\overline{\mathbb{F}}_{q}$ of $\mathbb{F}_{q}$. Then we denote the arithmetic Frobenius  $\overline{\mathbb{F}}_{q}\ni x\mapsto x^{q}\in \overline{\mathbb{F}}_{q}$ by $F$.  

Let $\mathbb{F}_{q^{p^{\infty}}}:= \bigcup_{n\geq 0} \mathbb{F}_{q^{p^{n}}}\subset \overline{\mathbb{F}}_{q}$. For each $n\geq 0$, put
\begin{align}
X_{n}:=X\times_{\mathbb{F}_{q}}{\rm Spec}(\mathbb{F}_{q^{p^{n}}}), ~~~X_{\infty}:=X\times_{\mathbb{F}_{q}}{\rm Spec}(\mathbb{F}_{q^{p^{\infty}}}), ~~~{\rm and}~\overline{X}:=X\times_{\mathbb{F}_{q}}{\rm Spec}(\overline{\mathbb{F}}_{q}). 
\end{align}
For $n\leq m$, we have $\mathbb{F}_{q^{p^{n}}}\subset \mathbb{F}_{q^{p^{m}}}$. Hence the natural projection $X_{m}\to X_{n}$ gives an inverse system $(X_{m}\to X_{n})_{n\leq m}$.  
\begin{prop}\label{Xinfty} (1) For $n\leq m$, the projection $X_{m}\to X_{n}$ is a finite \'etale morphism. 
\\(2) In the category of schemes, the scheme $X_{\infty}$ is the inverse limit of the inverse system $(X_{m}\to X_{n})_{n\leq m}$.
\end{prop}
\begin{proof}(1) Since the field extension $\mathbb{F}_{q^{p^n}}\subset \mathbb{F}_{q^{p^m}}$ is  finite and separable, the corresponding morphism ${\rm Spec}(\mathbb{F}_{q^{p^m}})\to {\rm Spec}(\mathbb{F}_{q^{p^n}})$ is \'etale and finite by \cite[I-Proposition 3.1]{milne}. Hence the base change $X_{m}=X_{n}\times_{\mathbb{F}_{q^{p^{n}}}}{\rm Spec}(\mathbb{F}_{q^{p^{m}}})\to X_{n}$ is a finite \'etale morphism by \cite[I-Proposition 1.3]{milne} and \cite[I-Proposition 3.3]{milne}. 

(2) Let $U={\rm Spec}(A)\subset X$ be an affine open set. Then, for every $n\geq 0$ or $n=\infty$, the pull-back of $U$ to $X_{n}$ is $U_{n}:={\rm Spec}(A\otimes_{\mathbb{F}_{q}}\mathbb{F}_{q^{p^{n}}})$. Since $\mathbb{F}_{q^{p^{\infty}}}=\varinjlim_{n}\mathbb{F}_{q^{p^{n}}}$, we have $A\otimes_{\mathbb{F}_{q}}\mathbb{F}_{q^{p^{\infty}}}=\varinjlim_{n}(A\otimes_{\mathbb{F}_{q}}\mathbb{F}_{q^{p^{n}}})$, thus $U_{\infty}=\varprojlim_{n}U_{n}$.

For a scheme $Y$, let $(\varphi_{n}: Y\to X_{n})_{n\geq 0}$ be a system of morphisms such that the composition $Y\to X_{m}\to X_{n}$ is $\varphi_{n}$ for every $n\leq m$. For each affine open set $U\subset X$, let $U_{Y}\subset Y$ be the pull-back of $U$ by $\varphi_{0}$. Then there exists a unique morphism $\varphi_{U}: U_{Y}\to U_{\infty}$ such that the composition $U_{Y}\to U_{\infty}\to U_{n}$ is $\varphi_{n}|_{U_{Y}}$ for all $n\geq 0$. For affine open sets $U, V\subset X$, the intersection $U\cap V$ is also affine. Hence we obtain a morphism $\varphi_{U\cap V}: (U\cap V)_{Y}=U_{Y}\cap V_{Y}\to (U\cap V)_{\infty}=U_{\infty}\cap V_{\infty}$. Since $\varphi_{U}|_{(U_{Y}\cap V_{Y})}=\varphi_{U\cap V}=\varphi_{V}|_{(U_{Y}\cap V_{Y})}$, this shows that there exists a unique morphism $\varphi: Y\to X_{\infty}$ such that the composition $Y\to X_{\infty}\to X_{n}$ is $\varphi_{n}$ for all $n\geq 0$, i.e., $X_{\infty}=\varprojlim_{n}X_{n}$.   
\end{proof}
\begin{defi}\label{brauer}(Cohomological Brauer groups.) For a smooth surface $S$ over an arbitrary field, we define the \emph{cohomological Brauer group} by 
\begin{align}
{\rm Br}(S):=H^{2}_{\text{\'et}}(S, \mathbb{G}_{m}),
\end{align}
where $\mathbb{G}_{m}$ is the \'etale sheaf over $S$ associated to the multiplicative group scheme on $S$. 
\end{defi}
\begin{rem} For $S$ in Definition \ref{brauer}, the \emph{Brauer group} of $S$ is the set of equivalence classes of Azumaya algebras with the group structure given by the tensor product. By \cite[Corollaire 2.2]{dix}, the Brauer group of $S$ is canonically isomorphic to the torsion part of $H^{2}_{\text{\'et}}(S, \mathbb{G}_{m})$. On the other hand, the group $H^{2}_{\text{\'et}}(S, \mathbb{G}_{m})$ is torsion by \cite[Chapter 18, Example 1.4-(ii)]{Huy}. Hence the Brauer group of $S$ coincides with the cohomological Brauer group ${\rm Br}(S)$. So we will call  ${\rm Br}(S)$ simply the Brauer group of $S$. 
\end{rem}
The inverse system $(X_{n}\to X_{m})_{n\geq m}$ gives a direct system of abelian groups 
\[({\rm Br}(X_{m})\to {\rm Br}(X_{n}))_{n\geq m}.\] 
\begin{prop}\label{nlim} The system of natural homomorphisms $({\rm Br}(X_{n})\to {\rm Br}(X_{\infty}))_{n\geq 0}$ is the direct limit of $({\rm Br}(X_{m})\to {\rm Br}(X_{n}))_{n\geq m},$ that is,
\begin{align}
{\rm Br}(X_{\infty})=\varinjlim_{n}{\rm Br}(X_{n}).
\end{align}
\end{prop}
\begin{proof}
This follows from Proposition \ref{Xinfty}-(2) and \cite[Expos\'e VII, Corollaire 5.8]{sga4}.
\end{proof}
\subsection{The Iwasawa modules ${\rm Br}(X_{\infty})[p^{\infty}]^{\lor}$ and ${\rm Br}(\overline{X})^{G^{\infty}}[p^{\infty}]^{\lor}$}  Now we introduce two $\mathbb{Z}_{p}$-modules ${\rm Br}(X_{\infty})[p^{\infty}]^{\lor}$ and ${\rm Br}(\overline{X})^{G^{\infty}}[p^{\infty}]^{\lor}$. We will show that these have natural structures of Iwasawa modules later.  

By Madapusi Pera \cite{mp15} and Kim--Madapusi Pera \cite{kmp}, the Tate conjecture for K3 surfaces holds for every positive characteristic. The following theorem is a consequence.
\begin{thm}[Madapusi Pera \cite{mp15} and Kim--Madapusi Pera \cite{kmp}]\label{tate conjecture for k3} For $n\geq 0$, the group ${\rm Br}(X_{n})[p^{\infty}]$ is a finite group. \end{thm}
\begin{proof}
By \cite[Chapter 18, Lemma 2.5]{Huy}, the finiteness of ${\rm Br}(X_{n})[p^{\infty}]$ is equivalent to the Tate conjecture for $H^{2}_{\text{\'et}}(\overline{X}, \mathbb{Z}_{p}(1))^{G^{n}}$. However, the Tate conjecture for K3 surfaces in every positive characteristic is true. 
\end{proof}
Let $G:={\rm Gal}(\overline{\mathbb{F}}_{q}/\mathbb{F}_{q})$ be the absolute Galois group. For $n\geq 0$, let $G^{n}:={\rm Gal}(\overline{\mathbb{F}}_{q}/\mathbb{F}_{q^{p^{n}}})$, and let ${\rm Br}(\overline{X})^{G^{n}}$ denote the $G^{n}$-invariant part of ${\rm Br}(\overline{X})$. The finiteness of the geometric parts ${\rm Br}(\overline{X})^{G^{n}}[p^{\infty}]$ has been proved by Colliot-Th\`el\'ene and Skorobogatov \cite{cs}.
\begin{thm}[Colliot-Th\`el\'ene and Skorobogatov \cite{cs}]\label{ls} For each $n\geq 0$, the group ${\rm Br}(\overline{X})^{G^{n}}[p^{\infty}]$ is a finite group.\end{thm}
\begin{proof}
   This follows from Theorem \ref{tate conjecture for k3} and \cite[Theorem 16.1.4]{cs}.
\end{proof}
\begin{lem} (1) The module ${\rm Br}(X_{\infty})[p^{\infty}]$ has a natural structure of a discrete $\Gamma$-module. In particular, the Pontryagin dual 
\begin{align}{\rm Br}(X_{\infty})[p^{\infty}]^{\lor}:={\rm Hom}_{{\rm cont}}({\rm Br}(X_{\infty})[p^{\infty}], \mathbb{Q}_{p}/\mathbb{Z}_{p})
\end{align} has a natural structure of a compact $\Lambda$-module, i.e., an Iwasawa module.
\\(2) The module 
\[{\rm Br}(\overline{X})^{G^{\infty}}[p^{\infty}]:=\varinjlim_{n}{\rm Br}(\overline{X})^{G^{n}}[p^{\infty}]=\bigcup_{n\geq 0}{\rm Br}(\overline{X})^{G^{n}}[p^{\infty}]\subset {\rm Br}(\overline{X})[p^{\infty}]\]
 has a natural structure of a discrete $\Gamma$-module. In particular,  the Pontryagin dual 
\begin{align}{\rm Br}(\overline{X})^{G^{\infty}}[p^{\infty}]^{\lor}:={\rm Hom}_{{\rm cont}}({\rm Br}(\overline{X})^{G^{\infty}}[p^{\infty}], \mathbb{Q}_{p}/\mathbb{Z}_{p})
\end{align} has a natural structure of a compact $\Lambda$-module, i.e., an Iwasawa module.
\end{lem}
\begin{proof}(1) By \cite[Theorem 5.3]{pont} (see also \cite[Theorem 1.1.8]{nsw}), the Pontryagin dual of a discrete module is a compact $\mathbb{Z}_{p}$-module with the compact-open topology. 

We see that the natural action of $\Gamma$ on the discrete group ${\rm Br}(X_{\infty})$ is continuous. By Proposition \ref{nlim}, we have
\[{\rm Br}(X_{\infty})[p^{\infty}]=\varinjlim_{n}\left({\rm Br}(X_{n})[p^{\infty}]\right).\]
 Let $\alpha\in {\rm Br}(X_{\infty})[p^{\infty}]$. It suffices to show that the stabilizer ${\rm Stab}(\alpha)$ is open in $\Gamma$. For some $n\geq 0$, $\alpha$ is in the image of ${\rm Br}(X_{n})[p^{\infty}]$, which is contained in ${\rm Br}(X_{\infty})^{\Gamma^{n}}[p^{\infty}]$, where $\Gamma^{n}:={\rm Gal}(\mathbb{F}_{q^{p^{\infty}}}/\mathbb{F}_{q^{p^{n}}})$. This means that $\Gamma^{n}\subset {\rm Stab}(\alpha)$, i.e., open. 
 
 For $g\in \Gamma$ and $f\in {\rm Br}(X_{\infty})[p^{\infty}]^{\lor}$, we define $gf(\alpha):=f(g^{-1}\alpha)$ for every $\alpha\in {\rm Br}(X_{\infty})[p^{\infty}]$. We  see that this action is continuous. By Lemma \ref{natiwasawa}, the compact $\mathbb{Z}_{p}$-module ${\rm Br}(X_{\infty})[p^{\infty}]^{\lor}$ has a natural structure of a compact $\Lambda$-module.

(2) As in (1), we can show that the Pontryagin dual ${\rm Br}(\overline{X})^{G^{\infty}}[p^{\infty}]^{\lor}$ has a natural structure of a compact $\mathbb{Z}_{p}$-module. 

We define the action of $\Gamma$ on ${\rm Br}(\overline{X})^{G^{\infty}}[p^{\infty}]$ as follows. Let $\gamma_{n}\in \Gamma_{n}$ be an element, and let $\alpha_{n}\in {\rm Br}(\overline{X})^{G^{n}}[p^{\infty}]$. Let $g_{n}\in G={\rm Gal}(\overline{\mathbb{F}}_{q}/\mathbb{F}_{q})$ be an  extension of $\gamma_{n}$ to $\overline{\mathbb{F}}_{q}$. Then we define  $\gamma_{n}\alpha_{n}:=g_{n}\alpha_{n}$. Since the group $G$ is commutative, the element $g_{n}\alpha_{n}$ is in ${\rm Br}(\overline{X})^{G^{n}}[p^{\infty}]$. If $g^{'}_{n}\in G$ is another extension of $\gamma_{n}$, then $g^{'}_{n}g^{-1}_{n}\in G^{n}$, thus $g^{'}_{n}g^{-1}_{n}\alpha_{n}=\alpha_{n}$. Hence $g^{'}_{n}\alpha_{n}=g_{n}\alpha_{n}$. 

For $(\gamma_{n})\in \Gamma=\varprojlim_{n}\Gamma_{n}$ and $\alpha\in {\rm Br}(\overline{X})^{G^{\infty}}[p^{\infty}]$, let 
\begin{equation}\label{gamman}
(\gamma_{n})\alpha:=\gamma_{n}\alpha_{n},
\end{equation}
 where we take $\alpha=\alpha_{n}\in {\rm Br}(\overline{X})^{G^{n}}[p^{\infty}]$ for some $n\geq 0$. For some $m\geq n$, if $\alpha_{m}$ is another element with $\alpha=\alpha_{m}$, then we  choose $g\in G$ as the extension of $\gamma_{m}$ to $\overline{\mathbb{F}}_{q}$. Since $g|_{\mathbb{F}_{q^{p^n}}}=\gamma_{n}$, we have $\gamma_{m}\alpha_{m}=g\alpha=\gamma_{n}\alpha_{n}$. 
 
 The action \eqref{gamman} is continuous. In fact, if we choose $\alpha=\alpha_{n}$, then we can see, in the same way as in (1), that $\Gamma^{n}\subset {\rm Stab}(\alpha)$. Hence, as in  (1), the dual ${\rm Br}(\overline{X})^{G^{\infty}}[p^{\infty}]^{\lor}$ has a natural structure of a compact $\Lambda$-module. \end{proof}
\subsection{The properties of the Iwasawa module ${\rm Br}(\overline{X})^{G^{\infty}}[p^{\infty}]^{\lor}$} Now, we shall prove that the Iwasawa module ${\rm Br}(\overline{X})^{G^{\infty}}[p^{\infty}]^{\lor}$ is a finitely generated torsion $\Lambda$-module. Let $\gamma\in\Gamma$ be the topological generator induced from the arithmetic Frobenius $F\in G$. 
 \begin{lem}\label{geoquo}For every $n\geq 0$, the projection ${\rm Br}(\overline{X})^{G^{\infty}}[p^{\infty}]^{\lor}\to {\rm Br}(\overline{X})^{G^{n}}[p^{\infty}]^{\lor}$ induces an isomorphism of $\Lambda/(\gamma^{p^{n}}-1)\Lambda$-modules
 \begin{align} 
{\rm Br}(\overline{X})^{G^{\infty}}[p^{\infty}]^{\lor}/(\gamma^{p^{n}}-1){\rm Br}(\overline{X})^{G^{\infty}}[p^{\infty}]^{\lor}\xrightarrow{\sim} {\rm Br}(\overline{X})^{G^{n}}[p^{\infty}]^{\lor}.
 \end{align}
 \end{lem}
 \begin{proof}Consider the exact sequence of $\Gamma$-modules:
 \[ 0\to {\rm Br}(\overline{X})^{G^{\infty}}[p^{\infty}]^{\gamma^{p^{n}}=1}\to {\rm Br}(\overline{X})^{G^{\infty}}[p^{\infty}]\xrightarrow{\gamma^{p^{n}}-1}{\rm Br}(\overline{X})^{G^{\infty}}[p^{\infty}],\]
 where ${\rm Br}(\overline{X})^{G^{\infty}}[p^{\infty}]^{\gamma^{p^{n}}=1}$ is the kernel of ${\rm Br}(\overline{X})^{G^{\infty}}[p^{\infty}]\xrightarrow{\gamma^{p^{n}}-1}{\rm Br}(\overline{X})^{G^{\infty}}[p^{\infty}]$. Since $F^{p^{n}}$ is a topological generator of $G^{n}$, we must have ${\rm Br}(\overline{X})^{G^{\infty}}[p^{\infty}]^{\gamma^{p^{n}}=1}={\rm Br}(\overline{X})^{G^{n}}[p^{\infty}]$. Taking the Pontryagin dual, we have an exact sequence of $\Lambda$-modules:
 \[{\rm Br}(\overline{X})^{G^{\infty}}[p^{\infty}]^{\lor}\xrightarrow{\gamma^{p^{n}}-1}{\rm Br}(\overline{X})^{G^{\infty}}[p^{\infty}]^{\lor}\to {\rm Br}(\overline{X})^{G^{n}}[p^{\infty}]^{\lor}\to0,\]
  where  the $\Lambda$-module structure of ${\rm Br}(\overline{X})^{G^{n}}[p^{\infty}]^{\lor}$ is induced by the restriction $\Gamma\to \Gamma_{n}$.
 \end{proof}
\begin{prop}\label{geotorsion}The Iwasawa module ${\rm Br}(\overline{X})^{G^{\infty}}[p^{\infty}]^{\lor}$ is a finitely generated torsion $\Lambda$-module. 
\end{prop}
\begin{proof}
In Lemma \ref{geoquo} for $n=0$, we have an isomorphism of $\Lambda/(\gamma-1)\Lambda$-modules \[{\rm Br}(\overline{X})^{G^{\infty}}[p^{\infty}]^{\lor}/(\gamma-1){\rm Br}(\overline{X})^{G^{\infty}}[p^{\infty}]^{\lor}\xrightarrow{\sim} {\rm Br}(\overline{X})^{G}[p^{\infty}]^{\lor}.
\]From Theorem \ref{ls}, the group ${\rm Br}(\overline{X})^{G}[p^{\infty}]$ is a finite group, and so its dual is also finite. Since $\gamma-1$ is an irreducible element in $\Lambda$,  Nakayama's lemma (for example, see  \cite[Corollary 2.26]{o1}) and \cite[Corollary 2.40]{o1} imply that the Iwasawa module ${\rm Br}(\overline{X})^{G^{\infty}}[p^{\infty}]^{\lor}$ is a finitely generated $\Lambda$-module. 
\end{proof} 
\begin{rem}
    We find that the equalities 
    \begin{equation}
        G^{n}=\{g^{p^{n}}~|~g\in G\}, ~\Gamma^{n}=\{g^{p^{n}}~|~ g\in \Gamma\}
    \end{equation}
    hold. For this reason, we use the superscript notations $G^{n}$ and $\Gamma^{n}$. 
\end{rem}
\section{Control Theorem}\label{control theorem}
In this section, we prove a K3 surface analogue of Mazur's control theorem (Theorem \ref{cont0}). 
\subsection{Tate modules} In this subsection, we define the L-function $L_{{\rm tr}}(X, t)$ mentioned in the introduction using the notion of  Tate modules. 
\begin{defi}\label{tatemodule}(Tate Modules). Let $X$ be a K3 surface over $\mathbb{F}_{q}$. 
For a prime $p\notmid q$, the \emph{Tate module} is the inverse limit
\begin{align}
T_{p}{\rm Br}(\overline{X}):=\varprojlim_{n}{\rm Br}(\overline{X})[p^{n}],
\end{align}
which is a free $\mathbb{Z}_{p}$-module. 
\end{defi}
Let $X$ and $p$ be as in Definition \ref{tatemodule}. For $n\geq 0$, the  Kummer exact sequence $0\to \mu_{p^{n}}\to \mathbb{G}_{m}\xrightarrow{p^{n}} \mathbb{G}_{m}\to 0$ yields a short exact sequence of $G={\rm Gal}(\overline{\mathbb{F}}_{q}/\mathbb{F}_{q})$-modules 
\begin{equation}\label{kummer exact}
0\to {\rm Pic}(\overline{X})\otimes\mathbb{Z}/p^{n}\mathbb{Z}\to H^{2}_{\text{\'et}}(\overline{X}, \mu_{p^{n}})\to {\rm Br}(\overline{X})[p^{n}]\to 0,
\end{equation}
where $\mu_{p^{n}}$ is the \'etale sheaf over $X$ defined by $U\mapsto \{\alpha\in \mathcal{O}_{U}(U)~|~\alpha^{p^{n}}=1\}$ for an \'etale morphism $U\to X$. Since the inverse system $({\rm Pic}(\overline{X})\otimes \mathbb{Z}/p^{n}\mathbb{Z})_{n\geq 0}$ satisfies the Mittag-Leffler condition (\cite[Section 10.86, Example 10.86.2]{stk}), we have 
\begin{equation}\label{tate exact}
0\to {\rm Pic}(\overline{X})\otimes\mathbb{Z}_{p}\to H^{2}_{\text{\'et}}(\overline{X}, \mathbb{Z}_{p}(1))\to T_{p}{\rm Br}(\overline{X})\to 0.
\end{equation}
In particular, if $\rho(\overline{X})$ is the Picard number of $\overline{X}$, then $T_{p}{\rm Br}(\overline{X})\cong\mathbb{Z}_{p}^{22-\rho(\overline{X})}$.
\begin{prop}\label{brtate} There exists a canonical $G$-isomorphism
\begin{equation}
{\rm Br}(\overline{X})[p^{\infty}]\cong T_{p}{\rm Br}(\overline{X})\otimes_{\mathbb{Z}_{p}}\mathbb{Q}_{p}/\mathbb{Z}_{p}.
\end{equation}
\end{prop}
\begin{proof} For each $m\geq 0$, we first show that the projection $H^{2}_{\text{\'et}}(\overline{X}, \mathbb{Z}_{p}(1))\to H^{2}_{\text{\'et}}(\overline{X}, \mu_{p^{m}})$ induces a natural $G$-isomorphism
\begin{equation}\label{uni}
H^{2}_{\text{\'et}}(\overline{X}, \mathbb{Z}_{p}(1))\otimes_{\mathbb{Z}_{p}}\mathbb{Z}_{p}/p^{m}\mathbb{Z}_{p}\xrightarrow{\sim}H^{2}_{\text{\'et}}(\overline{X}, \mu_{p^{m}}).
\end{equation}
Fix $n\geq 0$. Then there exists an exact sequence of \'etale sheaves on $X$:
\begin{equation}
0\to \mu_{p^n}\to \mu_{p^{m+n}}\xrightarrow{(-)^{p^{n}}}\mu_{p^m}\to 0.
\end{equation}
Hence we have a long exact sequence of $G$-modules:
\begin{equation}
H^{1}_{\text{\'et}}(\overline{X}, \mu_{p^m})\to H^{2}_{\text{\'et}}(\overline{X}, \mu_{p^n})\to H^{2}_{\text{\'et}}(\overline{X}, \mu_{p^{m+n}})\to H^{2}_{\text{\'et}}(\overline{X}, \mu_{p^m})\to H^{3}_{\text{\'et}}(\overline{X}, \mu_{\blue{p^n}})\to\cdots.
\end{equation}
Since $H^{1}_{\text{\'et}}(\overline{X}, \mu_{p^m})=0$ and $H^{3}_{\text{\'et}}(\overline{X}, \mu_{p^n})=0$, we have an exact sequence:
\begin{equation}
0\to H^{2}_{\text{\'et}}(\overline{X}, \mu_{p^n})\to H^{2}_{\text{\'et}}(\overline{X}, \mu_{p^{m+n}})\to H^{2}_{\text{\'et}}(\overline{X}, \mu_{p^m})\to 0.
\end{equation}
In particular, the inverse system $(H^{2}_{\text{\'et}}(\overline{X}, \mu_{p^n}))_{n\geq 0}$ is Mittag-Leffler (\cite[Section 10.86, Example 10.86.2]{stk}). Hence we have the following exact sequence:
\begin{equation}\label{uni1}
0\to H^{2}_{\text{\'et}}(\overline{X}, \mathbb{Z}_{p}(1))\xrightarrow{p^{m}}H^{2}_{\text{\'et}}(\overline{X}, \mathbb{Z}_{p}(1))\to H^{2}_{\text{\'et}}(\overline{X}, \mu_{p^m})\to 0, 
\end{equation}
where the homomorphism $H^{2}_{\text{\'et}}(\overline{X}, \mathbb{Z}_{p}(1))\xrightarrow{p^{m}}H^{2}_{\text{\'et}}(\overline{X}, \mathbb{Z}_{p}(1))$ is the composition of $H^{2}_{\text{\'et}}(\overline{X}, \mathbb{Z}_{p}(1))=\varprojlim_{n}H^{2}_{\text{\'et}}(\overline{X}, \mu_{p^{n}})\to \varprojlim_{n}H^{2}_{\text{\'et}}(\overline{X}, \mu_{p^{m+n}})$ and $\varprojlim_{n}H^{2}_{\text{\'et}}(\overline{X}, \mu_{p^{m+n}})\cong H^{2}_{\text{\'et}}(\overline{X}, \mathbb{Z}_{p}(1))$. The isomorphism \eqref{uni} follows from the exact sequence \eqref{uni1}. 

Using the Kummer sequence \eqref{kummer exact} and \eqref{uni1}, we have an exact sequence of $G$-modules:
\begin{equation}
0\to ({\rm Pic}(\overline{X})\otimes \mathbb{Z}_{p})\otimes_{\mathbb{Z}_{p}}\mathbb{Z}_{p}/p^{n}\mathbb{Z}_{p}\to H^{2}_{\text{\'et}}(\overline{X}, \mathbb{Z}_{p}(1))\otimes_{\mathbb{Z}_{p}}\mathbb{Z}_{p}/p^{n}\mathbb{Z}_{p}\to {\rm Br}(\overline{X})[p^{n}]\to 0.
\end{equation}
Taking the direct limit for $n\geq 0$, we have
\begin{equation}
0\to ({\rm Pic}(\overline{X})\otimes \mathbb{Z}_{p})\otimes_{\mathbb{Z}_{p}}\mathbb{Q}_{p}/\mathbb{Z}_{p}\to H^{2}_{\text{\'et}}(\overline{X}, \mathbb{Z}_{p}(1))\otimes_{\mathbb{Z}_{p}}\mathbb{Q}_{p}/\mathbb{Z}_{p}\to {\rm Br}(\overline{X})[p^{\infty}]\to 0,
\end{equation}
that is, using \eqref{tate exact}, ${\rm Br}(\overline{X})[p^{\infty}]\cong T_{p}{\rm Br}(\overline{X})\otimes_{\mathbb{Z}_{p}}\mathbb{Q}_{p}/\mathbb{Z}_{p}$.
\end{proof}
For every $n\geq 0$, the action of $G$ on $T_{p}{\rm Br}(\overline{X})^{G^{n}}$ induces the action of $\Gamma_{n}\cong G/G^{n}$. On the other hand, let $T_{p}{\rm Br}(\overline{X})^{\ast}:={\rm Hom}_{\mathbb{Z}_{p}}(T_{p}{\rm Br}(\overline{X}), \mathbb{Z}_{p})$ denote the dual. Then the natural action of $G$ on $T_{p}{\rm Br}(\overline{X})^{\ast}$ induces a continuous action on ${\rm coker}(1-F^{p^{n}}: T_{p}{\rm Br}(\overline{X})^{\ast}\to T_{p}{\rm Br}(\overline{X})^{\ast})$. Since $x=F^{p^{n}}x$ in ${\rm coker}(1-F^{p^{n}}: T_{p}{\rm Br}(\overline{X})^{\ast}\to T_{p}{\rm Br}(\overline{X})^{\ast})$, $G^{n}$ acts on ${\rm coker}(1-F^{p^{n}}: T_{p}{\rm Br}(\overline{X})^{\ast}\to T_{p}{\rm Br}(\overline{X})^{\ast})$ trivially. Hence the $G$-action induces the  $\Gamma_{n}$-action. For the same reason, the $G$-module ${\rm coker}(1-F^{p^{n}}:{\rm Br}(\overline{X})[p^{\infty}]^{\lor}\to {\rm Br}(\overline{X})[p^{\infty}]^{\lor})$ also has the structure of $\Gamma_{n}$-module.
\begin{prop}\label{dualtate}  There exists a canonical $\Gamma_{n}$-isomorphism
\begin{equation}
{\rm Br}(\overline{X})^{G^{n}}[p^{\infty}]^{\lor}\cong {\rm coker}(1-F^{p^{n}}: T_{p}{\rm Br}(\overline{X})^{\ast}\to T_{p}{\rm Br}(\overline{X})^{\ast}).
\end{equation}
\end{prop}
\begin{proof} 
By Proposition \ref{brtate}, we have an isomorphism of $G$-modules:
\begin{equation}
    {\rm Br}(\overline{X})[p^{\infty}]^{\lor}\cong (T_{p}{\rm Br}(\overline{X})\otimes_{\mathbb{Z}_{p}}\mathbb{Q}_{p}/\mathbb{Z}_{p})^{\lor}\cong T_{p}{\rm Br}(\overline{X})^{\ast}.
\end{equation}
Therefore, there is a canonical isomorphism of $\Gamma_{n}$-modules:
\begin{equation}
{\rm coker}(1-F^{p^{n}}:{\rm Br}(\overline{X})[p^{\infty}]^{\lor}\to {\rm Br}(\overline{X})[p^{\infty}]^{\lor})\cong{\rm coker}(1-F^{p^{n}}: T_{p}{\rm Br}(\overline{X})^{\ast}\to T_{p}{\rm Br}(\overline{X})^{\ast}).
\end{equation}
Since the $\mathbb{Z}$-module $\mathbb{Q}_{p}/\mathbb{Z}_{p}$ is divisible, the exact sequence of $G$-modules
\begin{equation}
    0\to {\rm Br}(\overline{X})^{G^{n}}[p^{\infty}]\to {\rm Br}(\overline{X})[p^{\infty}]\xrightarrow{1-F^{-p^{n}}} (1-F^{-p^{n}}){\rm Br}(\overline{X})[p^{\infty}]\to 0
\end{equation}
induces the exact sequence
\begin{equation}
    0\to (1-F^{p^{n}}){\rm Br}(\overline{X})[p^{\infty}]^{\lor}\to {\rm Br}(\overline{X})[p^{\infty}]^{\lor}\to {\rm Br}(\overline{X})^{G^{n}}[p^{\infty}]^{\lor}\to0.
\end{equation}
Thus 
\begin{equation}
    {\rm coker}(1-F^{p^{n}}:{\rm Br}(\overline{X})[p^{\infty}]^{\lor}\to {\rm Br}(\overline{X})[p^{\infty}]^{\lor})\cong {\rm Br}(\overline{X})^{G^{n}}[p^{\infty}]^{\lor}. 
\end{equation}
\end{proof}
\begin{defi}[The L-function of $T_{p}{\rm Br}(\overline{X})$]\label{Ltr} Let $X$ be a K3 surface over $\mathbb{F}_{q}$. Let $p\notmid q$ be a prime. The \emph{L-function} of $T_{p}{\rm Br}(\overline{X})$ is the polynomial 
\begin{equation}
L_{{\rm tr}}(X, t):={\rm det}(1-F^{-1}t: T_{p}{\rm Br}(\overline{X}))\in \mathbb{Z}_{p}[t].
\end{equation}
\end{defi}
\begin{cor}\label{geoeigen} We have
\begin{equation}
\#{\rm Br}(\overline{X})^{G^{n}}[p^{\infty}]=p^{v_{p}(L_{{\rm tr}}(X_{n}, 1))}.
\end{equation}
\end{cor}
\begin{proof}
Let $\#({\rm Br}(\overline{X})^{G^{n}}[p^{\infty}])=p^{m}, (m\geq 0)$. Then the characteristic ideal of the $\mathbb{Z}_{p}$-module ${\rm Br}(\overline{X})^{G_{n}}[p^{\infty}]$ is $(p^{m})$ by \eqref{simpchar}. By Proposition \ref{dualtate} and \cite[Proposition 4.3]{tu}, we have 
\begin{equation}
(p^{m})=({\rm det}(1-F^{p^{n}}: T_{p}{\rm Br}(\overline{X})^{\ast}))=(L_{{\rm tr}}(X_{n}, 1)).
\end{equation}
In particular, $m=v_{p}(p^{m})=v_{p}(L_{{\rm tr}}(X_{n}, 1))$.
\end{proof}
\subsection{The proof of the control theorem}
\begin{defi}[L-functions]\label{lfunction} The \emph{L-function} of $X/\mathbb{F}_{q}$ is the polynomial
\begin{equation} 
L(X, t):={\rm det}(1-F^{-1}t: H^{2}_{\text{\'et}}(\overline{X}, \mathbb{Q}_{p}(1)))\in \mathbb{Z}_{p}[t].
\end{equation}
\end{defi}
\begin{rem}
    
We briefly explain the relation between $L(X,t)$ and the zeta function $Z(X, t)$ of $X$. Since $F^{-1}$ acts on $\mathbb{Q}_p(1)$ as multiplication by $q^{-1}$, we have
\[
L(X, qt)={\rm det}\left(1-F^{-1}t:H^2_{\text{\'et}}(\overline{X},\mathbb{Q}_p)\right)\in\mathbb{Z}_p[t].
\]
For a K3 surface $X$, we have $H^{1}_{\text{\'et}}(\overline{X}, \mathbb{Q}_{p})=H^{3}_{\text{\'et}}(\overline{X}, \mathbb{Q}_{p})=0$. 
Thus, the rationality of the zeta function (for example, see \cite[Theorem 12.4]{milne}) says
\[
Z(X,t)=\frac{1}{(1-t)L(X, qt)(1-q^2t)}\in\mathbb{Q}_p(t).
\]
\indent Moreover, by the exact sequence \eqref{tate exact}, we have a decomposition of polynomials in $\mathbb{Z}_{p}[t]$
\begin{equation}L(X, t)={\rm det}(1-F^{-1}t: {\rm NS}(\overline{X})\otimes\mathbb{Q}_{p})\cdot L_{{\rm tr}}(X, t).\end{equation}
 By the Tate conjecture for K3 surfaces, the first factor ${\rm det}(1-F^{-1}t: {\rm NS}(\overline{X})\otimes\mathbb{Q}_{p})$ is precisely  the factor of $L(X, t)$ corresponding to the roots of unity.  Consequently, $L_{{\rm tr}}(X,t)$  is the factor corresponding to the roots that are not roots of unity.
\end{rem}
Fix an algebraic closure $\overline{\mathbb{Q}}_{p}$ of  $\mathbb{Q}_{p}$. Let $\beta_{1}, \cdots, \beta_{22}\in \overline{\mathbb{Q}}$ be the eigenvalues of the linear map $F^{-1}: H^{2}_{\text{\'et}}(\overline{X}, \mathbb{Q}_{p}(1))\to H^{2}_{\text{\'et}}(\overline{X}, \mathbb{Q}_{p}(1))$. Then we have 
\begin{equation}\label{eigen}
L(X, t)=\prod^{22}_{i=1}(1-\beta_{i}t).
\end{equation}
\begin{thm}[Artin--Tate conjecture for K3 surfaces]\label{artintate} Let $X/\mathbb{F}_{q}$ be a K3 surface. Then $\#{\rm Br}(X)$ is finite and 
\begin{equation}
\#{\rm Br}(X)\cdot|{\rm disc}\, {\rm NS}(X)|= q\cdot \prod_{i: \beta_{i}\neq1}(1-\beta_i),
\end{equation}
where ${\rm disc}\, {\rm NS}(X)$ is the discriminant of the lattice ${\rm NS}(X)$. 
\end{thm}
\begin{proof}
By \cite[Theorem 5.2]{tate}, this theorem follows from the Tate conjecture for K3 surfaces for some prime $p\notmid q$. However, the Tate conjecture for K3 surfaces is true for all prime $p\notmid q$ (\cite{mp15} and \cite{kmp}).  
\end{proof}
\begin{lem}\label{zeta} Fix an embedding $\overline{\mathbb{Q}}\hookrightarrow \overline{\mathbb{Q}}_{p}$. Let $\zeta\in \overline{\mathbb{Q}}$ be a root of unity. Let  $s=p^{a}m$, $m\in \mathbb{Z}, (p, m)=1$ be the order of $\zeta$. If for every integer $n\geq 0$ the number $1-\zeta^{p^{n}}$ is non-zero, then $v_{p}(1-\zeta^{p^{a+n}})=0$.
\end{lem}
\begin{proof} By the assumption, it follows that $s>1$. If $m=1$, then $a>0$ and $\zeta^{p^{a}}=1$. This is a contradiction. Thus $m>1$. Since $m$ is the order of $\zeta^{p^{n+a}}$, we have $1+\zeta^{p^{n+a}}+\cdots + (\zeta^{p^{n+a}})^{m-1}=0.$ On the other hand, we have
\begin{align}
m=(m-1)+1&=(m-1)-(\zeta^{p^{n+a}}+\cdots +(\zeta^{p^{n+a}})^{m-1})
\\&=(1-\zeta^{p^{n+a}})+\cdots+(1-(\zeta^{p^{n+a}})^{m-1})\in (1-\zeta^{p^{n+a}})\subset \overline{\mathbb{Z}}_{p},
\end{align}
where $\overline{\mathbb{Z}}_{p}$ is the integral closure of $\mathbb{Z}_{p}$ in $\overline{\mathbb{Q}}_{p}$. Since the integer $m$ is invertible in $\mathbb{Z}_{p}$, we have $v_{p}(1-\zeta^{p^{n+a}})=0$. 
\end{proof}
\begin{lem}\label{ngeo} For every sufficiently large $n\geq0$,  we have
    \begin{equation}
        v_{p}(\prod_{i, \beta_{i}\notin W_{n}}(1-\beta_{i}^{p^{n}}))=v_{p}(L_{{\rm tr}}(X_{n}, 1)),
    \end{equation}
    where $W_{n}$ is the set of $p^{n}$-th roots of unity.
\end{lem}
\begin{proof}
    Let $W^{'}$ be the set of $\beta_{i}$ which is  a root of unity. By the Tate conjecture, we have $L_{{\rm tr}}(X_{n}, T)=\prod_{i: \beta_{i}\notin W^{'}}(1-\beta_{i}^{p^{n}}T)$. Hence
    \begin{equation}
        \prod_{i, \beta_{i}\notin W_{n}}(1-\beta_{i}^{p^{n}})=\prod_{i: \beta_{i}\in W^{'}\backslash W_{n}}(1-\beta^{p^{n}}_{i})\cdot L_{{\rm tr}}(X_{n}, 1).
    \end{equation}
    Let $n_{0}$ be the maximal order of $\beta_{i}\in W^{'}$. Then, for $n>n_{0}$,  the element $\beta_{i}\in W^{'}\backslash W_{n}$ does not belong to $W:=\bigcup_{n\geq0}W_{n}$. By Lemma \ref{zeta}, we can find a large number $n_{1}\geq n_{0}$ such that for every $n\geq n_{1}$, we have $v_{p}(1-\beta^{p^{n}}_{i})=0$ for $\beta_{i}\in W^{'}\backslash W_{n}$. Thus
    \begin{equation}
        v_{p}(\prod_{i: \beta_{i}\in W^{'}\backslash W_{n}}(1-\beta^{p^{n}}_{i}))=0.
    \end{equation}
\end{proof}
\begin{thm}\label{geoarith}For every $n\geq 0$, the canonical $\Gamma_{n}$-homomorphism
\begin{equation}\label{geoarith1}
{\rm Br}(X_{n})[p^{\infty}]\to {\rm Br}(\overline{X})^{G^{n}}[p^{\infty}]
\end{equation}
has bounded kernel and cokernel. In particular,  the canonical $\Gamma$-homomorphism 
\begin{equation}\label{geoarith2}
{\rm Br}(X_{\infty})[p^{\infty}]\to {\rm Br}(\overline{X})^{G^{\infty}}[p^{\infty}]
\end{equation}
has finite kernel and cokernel. 

\end{thm}
\begin{proof}  Let
 \begin{align}
 K_{n}&:={\rm ker}({\rm Br}(X_{n})[p^{\infty}]\to {\rm Br}(\overline{X})^{G^{n}}[p^{\infty}]),
 \\ C_n&:={\rm coker}({\rm Br}(X_{n})[p^{\infty}]\to {\rm Br}(\overline{X})^{G^{n}}[p^{\infty}]).
\end{align}
The proof has two steps.

{\rm Step 1}.  We prove the boundedness of $\#K_{n}$. 

Using  the Hochschild--Serre spectral sequence (\cite[III-Theorem 2.20]{milne}):
\begin{equation}
E^{p, q}_{2}=H^{p}(G^{n}, H^{q}_{\text{\'et}}(\overline{X}, \mathbb{G}_{m}))\Rightarrow H^{p+q}_{\text{\'et}}(X_{n}, \mathbb{G}_{m})
\end{equation} 
and Hilbert's theorem 90, we have an exact sequence of $\mathbb{Z}$-modules
\begin{equation}\label{hsexact}
0\to {\rm Pic}(X_{n})[p^{\infty}]\to {\rm Pic}(\overline{X})^{G^{n}}[p^{\infty}]\to {\rm Br}(\mathbb{F}_{q^{p^{n}}})[p^{\infty}]\to K_{n} \to 
 H^{1}(G^{n}, {\rm Pic}(\overline{X}))[p^{\infty}].
\end{equation}
Since ${\rm Br}(\mathbb{F}_{q^{p^{n}}})=0$ (\cite[XIII-Proposition 5]{serre}), we have 
\begin{equation}
K_{n}\hookrightarrow H^{1}(G^{n}, {\rm Pic}(\overline{X}))[p^{\infty}].
\end{equation}
Hence it suffices to show that $\#H^{1}(G^{n}, {\rm Pic}(\overline{X}))$ is bounded. Let $N_{n}\subset G^{n}$ be the kernel of the Artin representation $G^{n}\to {\rm GL}({\rm Pic}(\overline{X}))$. Then the quotient $G^{n}/N_{n}$ is a finite group. By the five-term exact sequence (\cite[Proposition 1.6.6]{nsw}), we have an exact sequence of $\mathbb{Z}$-modules:
\begin{equation}\label{5term}
0\to H^{1}(G^{n}/N_{n}, {\rm Pic}(\overline{X}))\to H^{1}(G^{n}, {\rm Pic}(\overline{X}))\to H^{1}(N_{n}, {\rm Pic}(\overline{X}))^{G^{n}/N_{n}}.
\end{equation}
Since $N_{n}$ acts on ${\rm Pic}(\overline{X})$ trivially, we must have $H^{1}(N_{n}, {\rm Pic}(\overline{X}))={\rm Hom}_{{\rm cont}}(N_{n}, {\rm Pic}(\overline{X}))$. On the other hand,  the group $N_{n}$ is a profinite group, thus compact. Since the group ${\rm Pic}(\overline{X})$ is a discrete group, the images  of  all continuous homomorphisms from $N_{n}$ are finite subgroups of ${\rm Pic}(\overline{X})$. However, since ${\rm Pic}(\overline{X})$ is a free $\mathbb{Z}$-module, every finite subgroup is 0, i.e.,  ${\rm Hom}_{{\rm cont}}(N_{n}, {\rm Pic}(\overline{X}))=0$. Hence the exact sequence \eqref{5term} says 
\begin{equation}
H^{1}(G^{n}/N_{n}, {\rm Pic}(\overline{X}))\cong H^{1}(G^{n}, {\rm Pic}(\overline{X})).
\end{equation}
Let $s\geq 0$ be the order of the generator $F^{p^{n}}\in G^{n}/N_{n}$. By \cite[IV-Section 8]{cassels}, the left-hand side is isomorphic to the quotient
\begin{equation}\label{groupcoh}
\frac{{\rm ker}(1+F^{p^{n}}+F^{2p^{n}}+\cdots+F^{(s-1)p^{n}}: {\rm Pic}(\overline{X})\to {\rm Pic}(\overline{X}))}{(F^{p^{n}}-1){\rm Pic}(\overline{X})}. 
\end{equation}
On the other hand, let $r\geq 0$ be the order of the automorphism $F\in {\rm GL}({\rm Pic}(\overline{X}))$. Then, for some $0\leq i\leq r$, the automorphism $F^{p^{n}}\in {\rm GL}({\rm Pic}(\overline{X}))$ is equal to $F^{i}$. Note that $s$ is equal to the order of $F^{i}\in {\rm GL}({\rm Pic}(\overline{X}))$. Hence the group \eqref{groupcoh} is one of the groups 
\begin{equation}
\frac{{\rm ker}(1+F^{i}+F^{2i}+\cdots+F^{(s-1)i}: {\rm Pic}(\overline{X})\to {\rm Pic}(\overline{X}))}{(F^{i}-1){\rm Pic}(\overline{X})},
\end{equation}
where $0\leq i\leq r-1$. Since the group $H^{1}(G^{n}, {\rm Pic}(\overline{X}))$ is finite (\cite[Chapter 18, Lemma 2.2]{Huy}), this shows Step 1. 

{\rm Step 2}. For every sufficiently large $n\geq 0$, the quotient 
\begin{equation}\label{quo}
\frac{\#{\rm Br}(\bar X)^{G^n}[p^{\infty}]}{\#{\rm Br}(X_{n})[p^{\infty}]}
\end{equation}
is a constant integer. In particular, the quotient \eqref{quo} is bounded. 

Consider an increasing sequence of submodules:
\begin{equation}
{\rm NS}(X)\subset {\rm NS}(X_{1})\subset \cdots \subset {\rm NS}(\overline{X}).
\end{equation}
Since ${\rm NS}(\overline{X})$ is a finitely generated $\mathbb{Z}$-module, there exists a number $n_{0}\geq 0$ such that for $n\geq n_{0}$ we have ${\rm NS}(X_{n_{0}})={\rm NS}(X_{n})$. 

By Corollary \ref{geoeigen} and Lemma \ref{ngeo}, for sufficiently large $n\geq0$, we have
\begin{equation}
v_{p}(\#{\rm Br}(\overline{X})^{G^{n}}[p^{\infty}])
=
\sum_{i: \beta_{i}\notin W_{n}}v_{p}(1-\beta_{i}^{p^{n}}).
\end{equation}
By the Artin--Tate formula applied to $X_{n}$, we have
\[v_{p}(\#{\rm Br}(X_{n})[p^{\infty}])+ v_{p}(|{\rm disc}\,{\rm NS}(X_{n})|)=v_{p}(\prod_{i: \beta_{i}\notin W_{n}}(1-\beta^{p^{n}}_{i}))=v_{p}(\#{\rm Br}(\overline{X})^{G^{n}}[p^{\infty}]),\]
that is, for every $n\geq n_{0}$, 
\begin{equation}
v_{p}\left(\frac{\#{\rm Br}(\overline{X})^{G^{n}}[p^{\infty}]}{\#{\rm Br}(X_{n})[p^{\infty}]}\right)=v_{p}(|{\rm disc} {\rm NS}(X_{n})|)=v_{p}(|{\rm disc} {\rm NS}(X_{n_{0}})|).
\end{equation} 
This shows Step 2. 

Since
\begin{equation}
\# C_n
=
\frac{
\#\operatorname{Br}(\bar X)^{G_n}[p^{\infty}]\cdot \#K_{n}
}{
\#\operatorname{Br}(X_{n})[p^{\infty}]
}
\end{equation}
for every sufficiently large $n\geq 0$, the order of the cokernel $\# C_{n}$ is bounded by Steps 1 and 2. 
\end{proof}
\begin{cor}\label{tormod} The Iwasawa module ${\rm Br}(X_{\infty})[p^{\infty}]^{\lor}$ is a finitely generated torsion $\Lambda$-module.  
\end{cor}
\begin{proof}
By Theorem \ref{geoarith}, the dual of the homomorphism ${\rm Br}(X_{\infty})[p^{\infty}]\to {\rm Br}(\overline{X})^{G^{\infty}}[p^{\infty}]$ also has finite kernel and cokernel. Since the Iwasawa module ${\rm Br}(\overline{X})^{G^{\infty}}[p^{\infty}]^{\lor}$ is a finitely generated torsion $\Lambda$-module (Proposition \ref{geotorsion}), so is the Iwasawa module ${\rm Br}(X_{\infty})[p^{\infty}]^{\lor}$. 
\end{proof}
\begin{lem}\label{tran} For each integer $n\geq 0$, let 
\begin{equation}
A_{n}\xrightarrow{f_{n}} B_{n}\xrightarrow{g_{n}} C_{n}
\end{equation}
be a sequence of morphisms of modules. If $g_{n}\circ f_{n}$ and $g_{n}$ have kernels and cokernels of bounded order, then $f_{n}$ has bounded kernel and cokernel. 
\end{lem}
\begin{proof}Clearly, ${\rm ker}(f_{n})\subset {\rm ker}(g_{n}\circ f_{n})$. So ${\rm ker}(f_{n})$ is bounded by assumption. On the other hand, $g_{n}$ induces an embedding:
\begin{equation}
\frac{{\rm coker}(f_{n})}{\langle {\rm ker}(g_{n}), {\rm im}(f_{n})\rangle/{\rm im}(f_{n})}\hookrightarrow {\rm coker}(g_{n}\circ f_{n}),
\end{equation}
where $\langle{\rm ker}(g_{n}), {\rm im}(f_{n})\rangle$ is the submodule generated by ${\rm ker}(g_{n})$ and ${\rm im}(f_{n})$ in $B_{n}$. Since ${\rm ker}(g_{n})$ and ${\rm coker}(g_{n}\circ f_{n})$ are bounded, it follows that ${\rm coker}(f_{n})$ is bounded.
\end{proof}
\begin{thm}[Control Theorem]\label{K3conto} For every $n\geq 0$, the natural homomorphism of $\Gamma_{n}$-modules
\begin{equation}
{\rm Br}(X_{n})[p^{\infty}]\to {\rm Br}(X_{\infty})^{\Gamma^{n}}[p^{\infty}]
\end{equation}
has bounded kernel and cokernel. 
\end{thm}
\begin{proof} Let $f: {\rm Br}(X_{\infty})[p^{\infty}]\to {\rm Br}(\overline{X})^{G^{\infty}}[p^{\infty}]$ be the morphism \eqref{geoarith2}. Since the composition via projections $\overline{X}\to X_{\infty}\to X_{n}$ is equal to the projection $\overline{X}\to X_{n}$, $f$ factors as 
\begin{equation}
{\rm Br}(X_{n})[p^{\infty}]\to {\rm Br}(X_{\infty})^{\Gamma^{n}}[p^{\infty}]\xrightarrow{f^{\Gamma^{n}}} ({\rm Br}(\overline{X})^{G^{\infty}}[p^{\infty}])^{\Gamma^{n}}={\rm Br}(\overline{X})^{G^{n}}[p^{\infty}],
\end{equation}
where $f^{\Gamma^{n}}$ is the morphism induced by the left-exact functor $(-)^{\Gamma^{n}}$ from $\Gamma^{n}$-modules to $\mathbb{Z}$-modules. 

To use Lemma \ref{tran}, we show that the morphism $f^{\Gamma^{n}}$ has bounded kernel and cokernel. In fact, since ${\rm ker}(f^{\Gamma^{n}})={\rm ker}(f)^{\Gamma^{n}}\subset {\rm ker}(f)$,  the kernel of $f^{\Gamma^{n}}$ is bounded. On the other hand, it is easy to see the inclusion ${\rm im}(f^{\Gamma^{n}})\subset {\rm im}(f)^{\Gamma^{n}}$. Moreover, the exact sequence of $\Gamma$-modules
\begin{equation}
0\to {\rm im}(f)\to {\rm Br}(\overline{X})^{G^{\infty}}[p^{\infty}]\to {\rm coker}(f)\to 0
\end{equation}
induces an exact sequence of $\mathbb{Z}$-modules
\begin{equation}
0\to {\rm im}(f)^{\Gamma^{n}}\to {\rm Br}(\overline{X})^{G^{n}}[p^{\infty}]\to {\rm coker}(f)^{\Gamma^{n}}.
\end{equation}
Thus we obtain
\begin{equation}{\rm Br}(\overline{X})^{G^{n}}[p^{\infty}]/{\rm im}(f)^{\Gamma^{n}}\hookrightarrow {\rm coker}(f)^{\Gamma^{n}}\hookrightarrow {\rm coker}(f).
\end{equation}  
Since ${\rm coker}(f^{\Gamma^{n}})$ is a quotient of ${\rm Br}(\overline{X})^{G^{n}}[p^{\infty}]/{\rm im}(f)^{\Gamma^{n}}$, the boundedness of the cokernel of $f^{\Gamma^{n}}$ follows from the boundedness of ${\rm coker}(f)$. 

By Lemma \ref{tran}, the boundedness of $f$ and $f^{\Gamma^{n}}$ implies the boundedness of ${\rm Br}(X_{n})[p^{\infty}]\to {\rm Br}(X_{\infty})^{\Gamma^{n}}[p^{\infty}]$.
\end{proof}
\section{Iwasawa-type formula for ${\rm Br}(X_{n})[p^{\infty}]$}\label{iwaproof}
In this section, we prove Theorem \ref{iwasawatype}. We give two independent proofs of this Iwasawa-type formula for ${\rm Br}(X_{n})[p^{\infty}]$. 
\subsection{First proof of Iwasawa-type formula}
The first proof is based on the proof of Theorem \ref{geoarith} and module-theoretic Iwasawa-type formula (Theorem \ref{moduiwasawa}). 
\begin{thm}[Iwasawa-type formula]\label{iwasawatype2}  Let $\mu=\mu({\rm Br}(X_{\infty})[p^{\infty}]^{\lor})$ and let $ \lambda=\lambda({\rm Br}(X_{\infty})[p^{\infty}]^{\lor})$ be the Iwasawa invariants of the Iwasawa module ${\rm Br}(X_{\infty})[p^{\infty}]^{\lor}$ (see Definition \ref{moduinv}). Then
 there exists a unique integer $\nu=\nu({\rm Br}(X_{\infty})[p^{\infty}]^{\lor})$ such that, for every sufficiently large $n\geq 0$, we have
\[
\# {\rm Br}(X_n)[p^{\infty}]=p^{\mu p^n+\lambda n+\nu}.
\] 	
\end{thm}
\begin{proof}  Let $\mu_{0}, \lambda_{0}\in\mathbb{Z}$ be the Iwasawa invariants for ${\rm Br}(\overline{X})^{G^{\infty}}[p^{\infty}]^{\lor}$. By Lemma \ref{geoquo} and Theorem \ref{moduiwasawa}, there exists a unique $\nu_{0}\in\mathbb{Z}$ such that, for every sufficiently large $n\geq 0$, we have 
\begin{equation}
\#{\rm Br}(\overline{X})^{G^{n}}[p^{\infty}]=\#\frac{{\rm Br}(\overline{X})^{G^{\infty}}[p^{\infty}]^{\lor}}{(\gamma^{p^{n}}-1){\rm Br}(\overline{X})^{G^{\infty}}[p^{\infty}]^{\lor}}=p^{\mu_{0} p^{n}+\lambda_{0} n+\nu_{0}},
\end{equation}
where $\gamma\in \Gamma$ is the topological generator induced from the arithmetic Frobenius $F$. 
Since ${\rm Br}(X_{\infty})[p^{\infty}]^{\lor}$ is pseudo-isomorphic to ${\rm Br}(\overline{X})^{G^{\infty}}[p^{\infty}]^{\lor}$, by Theorem \ref{iwstr}, we have   $\mu=\mu_{0}$ and $\lambda=\lambda_{0}$.
By Step 2 of Theorem \ref{geoarith}, there exists an integer $\nu_{1}$ such that, for every sufficiently large $n\geq 0$, we have
\begin{equation}
\frac{\#{\rm Br}(\overline{X})^{G^{n}}[p^{\infty}]}{\#{\rm Br}(X_{n})[p^{\infty}]}=p^{v_{1}}.
\end{equation}
Hence
\begin{equation}
\# {\rm Br}(X_n)[p^{\infty}]=\#{\rm Br}(\overline{X})^{G^{n}}[p^{\infty}]p^{-\nu_{1}}=p^{\mu p^n+\lambda n+\nu},
\end{equation}
where $\nu:=\nu_{0}-\nu_{1}$. 
\end{proof}
\subsection{Second proof of Iwasawa-type formula}

The second proof uses the fundamental evaluation formula in Iwasawa theory (Theorem \ref{fef}).

Fix an embedding $\overline{\mathbb{Q}}\hookrightarrow\overline{\mathbb{Q}}_{p}$. For a given K3 surface $X/\mathbb{F}_{q}$,  let $L(X, t)=\prod_{i=1}^{22}(1-\beta_i t)\in\mathbb{Z}_{p}[t]$ as before.

\begin{thm}[Iwasawa-type formula]\label{iwasawatype1}  Let $X/\mathbb{F}_{q}$ be a K3 surface, and let $X_{n}:=X\times_{\mathbb{F}_{q}}{\rm Spec}(\mathbb{F}_{q^{p^{n}}})$. Let $\mu=\mu(L_{{\rm tr}}(X, 1+T))$ and $ \lambda=\lambda(L_{{\rm tr}}(X, 1+T))$ be the Iwasawa invariants of $L_{{\rm tr}}(X, 1+T)$ (see Definition \ref{polyiwa}). Then there exists a unique integer $\nu=\nu(L_{{\rm tr}}(X, 1+T))$ such that, for every sufficiently large $n\geq 0$, we have
\[
\# {\rm Br}(X_n)[p^{\infty}]=p^{\mu p^n+\lambda n+\nu}.
\] 	
\end{thm}
\begin{proof} By the Tate conjecture for K3 surfaces, we have $L_{{\rm tr}}(\zeta)\neq0$ for all $\zeta\in W\backslash\{1\}$, where $W:=\bigcup_{n\geq0}W_{n}$. Moreover,  
\[
\prod_{\zeta\in W_n}L_{{\rm tr}}(X, \zeta)=\prod_{i: \beta_{i}\notin W^{'}}\prod_{\zeta\in W_n}(1-\zeta\beta_i)=\prod_{i: \beta_{i}\notin W^{'}}(1-\beta_i^{p^n})=L_{{\rm tr}}(X_{n}, 1).
\]
On the other hand, by the Artin\blue{--}Tate conjecture for K3 surfaces (Theorem \ref{artintate}), we have
\[
\#{\rm Br}(X_n)\cdot |{\rm disc}\, {\rm NS}(X_n)|= q^{p^n}\cdot \prod_{i: \beta_{i}\notin W_{n}}(1-\beta_i^{p^n}).
\]

Since the $\mathbb{Z}$-module ${\rm NS}(\overline{X})$ is free, the number $|{\rm disc}\, {\rm NS}(X_{n})|$ is  constant for all sufficiently large $n\geq 0$. Let $n_{0}$ be a  sufficiently large integer obtained from  fundamental evaluation formula (Theorem \ref{fef}) and Lemma \ref{ngeo}. Therefore, for sufficiently large $n>n_{0}$, we have 
\begin{align}
v_p(\#\mbox{Br}(X_n))&=v_p(\prod_{i: \beta_{i}\notin W_{n}}(1-\beta_i^{p^n}))-v_{p}({\rm disc}\,{\rm NS}(X_{n}))\\
&=v_{p}(L_{{\rm tr}}(X_{n}, 1))-v_{p}({\rm disc}\,{\rm NS}(X_{n}))
\\
&=v_p(\prod_{\zeta\in W_n\backslash\{1\}}L_{{\rm tr}}(X, \zeta))+v_{p}(L_{{\rm tr}}(X, 1))-v_{p}({\rm disc}\,{\rm NS}(X_{n}))=\mu p^n+\lambda n+\nu,
\end{align}
where $\nu:=v_{p}(L_{{\rm tr}}(X, 1))-v_{p}({\rm disc}\,{\rm NS}(X_{n}))$ is a constant integer. 
\end{proof}
\begin{cor}\label{coinv} Let $X/\mathbb{F}_{q}$ be a K3 surface, and let $X_{\infty}:=X\times_{\mathbb{F}_{q}}{\rm Spec}(\mathbb{F}_{q^{p^{\infty}}})$. Then the following statements hold.
  \\(1) $\mu({\rm Br}(X_{\infty})[p^{\infty}]^{\lor})= \mu(L_{{\rm tr}}(X, 1+T))=0$. 
  \\(2) $\lambda({\rm Br}(X_{\infty})[p^{\infty}]^{\lor})=\lambda(L_{{\rm tr}}(X, 1+T))$.
  \\(3) $\nu({\rm Br}(X_{\infty})[p^{\infty}]^{\lor})=\nu(L_{{\rm tr}}(X, 1+T))=v_{p}(L_{{\rm tr}}(X, 1))-v_{p}({\rm disc}\,{\rm NS}(X_{n}))$ where $n$ is a sufficiently large integer such that ${\rm NS}(X_{n})={\rm NS}(X_{n+1})=\cdots$.
\end{cor}
\begin{proof} It suffices to show that $\mu(L_{{\rm tr}}(X, 1+T))$ is 0. Since
\begin{equation}
    (1+T)^{-1}=1-T+T^{2}-\cdots,
\end{equation}
we have 
\begin{equation}
    (1-\beta_{i}(1+T))=(1+T)(1-\beta_{i}-T+T^{2}-\cdots),
\end{equation}
where $\beta_{i}\notin W^{'}$. The polynomial $1+T$ is a unit of $\mathbb{Z}_{p}[[T]]$ and $p\notmid (1-\beta_{i}-T+T^{2}-\cdots)$. Hence we have $p\notmid \prod_{i: \beta_{i}\notin W^{'}}(1-\beta_{i}(1+T))=L_{{\rm tr}}(X, 1+T)$. This means $\mu(L_{{\rm tr}}(X, 1+T))=0$. 
\end{proof}
\subsection{Examples: Kummer surfaces associated to the products of two elliptic curves}
In this subsection, we compute the Iwasawa invariants for Kummer surfaces associated to the products of two non-isogenous elliptic curves. Throughout this subsection, we assume ${\rm Char}(\mathbb{F}_{q})\neq 2$. 

Let $p\notmid q$ be a prime. Let $E_{i}/\mathbb{F}_{q}, i=1, 2$, be elliptic curves, and let $A:=E_{1}\times E_{2}$. Let $-1: A\to A$ be the involution.  Then $-1$ has 16 fixed points $P_{1}, \cdots, P_{16}\in A(\overline{\mathbb{F}}_{q})$. The quotient $(A)/\langle-1\rangle$ has rational double point singularities at the images of the $P_{i}$. The Kummer surface $X:={\rm Km}(A)$ is the minimal resolution 
\begin{equation}\label{resolution}
    X\to A/\langle-1\rangle.
\end{equation}
\begin{lem}\label{abelalg} Suppose that $\overline{E}_{1}$ and $\overline{E}_{2}$ are not isogenous. Then $\rho(\overline{A})=2$ and we have
 ${\rm NS}(\overline{A})={\rm NS}(A)$.
\end{lem}
\begin{proof} Let $O_{i}$ be the origins of $E_{i}$ for $i=1, 2$. Let $\pi_{i}: \overline{A}\to \overline{E}_{i},  i=1, 2,$ be the projections. Then the pull-back $\pi_{i}^{\ast}: {\rm Pic}(\overline{E}_{i})\to {\rm Pic}(\overline{A})$  sends the ample class $\mathcal{O}_{\overline{E}_{1}}([O_{1}])$ to $\mathcal{O}_{\overline{A}}([\{O_{1}\}\times \overline{E}_{2}])$ (resp. $\mathcal{O}_{\overline{A}}([\overline{E}_{1}\times \{O_{2}\}])$). Let 
\begin{equation}\mathcal{L}:=\mathcal{O}_{\overline{A}}([\overline{E}_{1}\times \{O_{2}\}])\otimes\mathcal{O}_{\overline{A}}([\{O_{1}\}\times \overline{E}_{2}])\in {\rm Pic}(\overline{A}).
\end{equation}
By \cite[Lemma 28.27.15]{stk}, the invertible sheaf $\mathcal{L}$ is ample. 

For $\mathcal{M}\in {\rm Pic}(\overline{A})$ (resp. $\mathcal{M}_{i}\in {\rm Pic}(\overline{E}_{i})$), let $\varphi_{\mathcal{M}}: A(\overline{\mathbb{F}}_{q})\to {\rm Pic} (\overline{A})$ (resp. $\varphi_{\mathcal{M}_{i}}: E_{i}(\overline{\mathbb{F}}_{q})\to {\rm Pic}(\overline{E}_{i}), i=1, 2)$ be given by $a\mapsto t_{a}^{\ast}\mathcal{M}\otimes \mathcal{M}^{-1}$ (resp. $a\mapsto t_{a}^{\ast}\mathcal{M}_{i}\otimes \mathcal{M}_{i}^{-1}$), where $t_{a}$ is the translation by $a$. Then the images of $\varphi_{\mathcal{M}}$ and $\varphi_{\mathcal{M}_{i}}$ are contained in ${\rm Pic}^{0}(\overline{A})$ and ${\rm Pic}^{0}(\overline{E}_{i})$ (\cite[Proposition 10.1]{milne0}), and the following diagram is commutative:

\begin{equation}\label{ellipic}
\begin{tikzcd}[column sep =7em]
    {\rm Pic}^{0}(\overline{E}_{1})\times {\rm Pic}^{0}(\overline{E}_{2})\ar[r, "\pi_{1}^{\ast}(-)\otimes \pi_{2}^{\ast}(-)"] & 
    {\rm Pic}^{0}(\overline{A})
    \\
    \overline{A}\ar[r, equal]\ar[u, two heads, "{(\varphi_{\mathcal{M}_{1}}, \varphi_{\mathcal{M}_{2}})}"] & \overline{A}\ar[u, two heads, "\varphi_{\mathcal{M}}"'].
\end{tikzcd}
\end{equation}

Since ${\rm Hom}(\overline{E}_{i}, \overline{E_{j}})=0$ for $i\neq j$, we have ${\rm End}(\overline{A})\cong {\rm End}(\overline{E}_{1})\times {\rm End}(\overline{E}_{2})$. In particular, we have ${\rm End}^{0}(\overline{A})\cong {\rm End}^{0}(\overline{E}_{1})\times {\rm End}^{0}(\overline{E}_{2})$, where ${\rm End}^{0}(-):={\rm End}(-)\otimes_{\mathbb{Z}}\mathbb{Q}$. Moreover, by the diagram \eqref{ellipic}, we have
\begin{align}
\varphi_{L}\circ (\varphi_{O_{1}}^{-1}\circ \varphi_{\mathcal{M}_{1}}, \varphi^{-1}_{O_{2}}\circ \varphi_{\mathcal{M}_{2}})
=\otimes(\varphi_{O_{1}}, \varphi_{O_{2}})\circ (\varphi_{O_{1}}^{-1}\circ \varphi_{\mathcal{M}_{1}}, \varphi^{-1}_{O_{2}}\circ \varphi_{\mathcal{M}_{2}})
=\otimes (\varphi_{\mathcal{M}_{1}}, \varphi_{\mathcal{M}_{2}}).
\end{align}
for all $\mathcal{M}\in {\rm Pic}(\overline{A})$ and $\mathcal{M}_{i}\in{\rm Pic}(\overline{E}_{i})$. Hence the following diagram
\[
\begin{tikzcd}[column sep = 7em]
   {\rm End}^{0}(\overline{E}_{1})\times {\rm End}^{0}(\overline{E}_{2}) \ar[r, "\sim"] &{\rm End}^{0}(\overline{A})\\
     {\rm NS}(\overline{E}_{1})\otimes\mathbb{Q}\times {\rm NS}(\overline{E}_{2})\otimes\mathbb{Q} \ar[u, hook]\ar[r, "\pi_{1}^{\ast}(-)\otimes \pi_{2}^{\ast}(-)"']& {\rm NS}(\overline{A})\otimes\mathbb{Q}\ar[u, hook]
\end{tikzcd}
\]
is commutative, where the vertical arrows are defined by ${\rm Pic}(\overline{A})\ni \mathcal{M}\mapsto \varphi^{-1}_{\mathcal{L}}\circ \varphi_{\mathcal{M}} \in {\rm End}^{0}(\overline{A})$ (resp. ${\rm Pic}(\overline{E}_{i})\ni \mathcal{M}_{i}\mapsto \varphi^{-1}_{O_{i}}\circ \varphi_{\mathcal{M}_{i}}\in {\rm End}^{0}(\overline{E}_{i})$). Let $(-)^{\dagger}: {\rm End}^{0}(\overline{A})\to {\rm End}^{0}(\overline{A})$ (resp. $(-)^{\dagger}: {\rm End}^{0}(\overline{E}_{i})\to {\rm End}^{0}(\overline{E}_{i})$) be the Rosati involution corresponding to the polarization $\mathcal{L}$ (resp. $\mathcal{O}_{\overline{E}_{i}}(\{O_{i}\})$). Then the isomorphism ${\rm End}^{0}(\overline{E}_{1})\times {\rm End}^{0}(\overline{E}_{2})\xrightarrow{\sim}{\rm End}^{0}(\overline{A})$ induces an isomorphism ${\rm End}^{0}(\overline{E}_{1})^{\dagger}\times {\rm End}^{0}(\overline{E}_{2})^{\dagger}\xrightarrow{\sim}{\rm End}^{0}(\overline{A})^{\dagger}$, where ${\rm End}^{0}(-)^{\dagger}$ is the invariant part by $(-)^{\dagger}$. By \cite[Proposition 17.2]{milne0} and the diagram, the bottom horizontal morphism  ${\rm NS}(\overline{E}_{1})\otimes\mathbb{Q}\times {\rm NS}(\overline{E}_{2})\otimes \mathbb{Q}\to {\rm NS}(\overline{A})\otimes\mathbb{Q}$ is an isomorphism. In particular, the $\mathbb{Q}$-vector space ${\rm NS}(\overline{A})\otimes \mathbb{Q}$ is generated by $\mathcal{O}_{\overline{A}}(E_{1}\times \{O_{2}\})$ and $\mathcal{O}_{\overline{A}}(\{O_{1}\}\times E_{2})$. Since the module ${\rm NS}(\overline{A})$ is a free $\mathbb{Z}$-module (\cite[Corollary 12.8]{milne0}), we have ${\rm NS}(\overline{A})\hookrightarrow{\rm NS}(\overline{A})\otimes\mathbb{Q}={\rm NS}(\overline{A})^{G}\otimes \mathbb{\mathbb{Q}}$, and thus ${\rm NS}(\overline{A})={\rm NS}(\overline{A})^{G}={\rm NS}(A)$.
\end{proof}
\begin{prop}\label{nskummer}Assume that $\overline{E}_{1}$ and $\overline{E}_{2}$ are not isogenous. 
If $\overline{E}_{i}[2]:=\{P\in E_{i}(\overline{\mathbb{F}}_{q})|~2P=O\}\subset E_{i}(\mathbb{F}_{q})$, then we have ${\rm NS}(X)={\rm NS}(\overline{X})$. 
\end{prop}
\begin{proof}
    By the construction of Kummer surfaces, there exists an isomorphism 
    \begin{equation}\label{kummercoh}
        (H^{2}_{\text{\'et}}(\overline{X}, \mathbb{Q}_{p}(1)), (-, -)_{\overline{X}})\cong (H^{2}_{\text{\'et}}(\overline{A}, \mathbb{Q}_{p}(1)), 2(-, -)_{\overline{A}})\oplus \bigoplus^{16}_{i=1}\mathbb{Q}_{p}[D_{i}]
    \end{equation}
    preserving the action of $G={\rm Gal}(\overline{\mathbb{F}}_{q}/\mathbb{F}_{q})$, where $[D_{i}]\in {\rm NS}(\overline{X})$ are the cohomology classes of the exceptional divisors of the resolution \eqref{resolution} and $(-, -)_{\overline{X}}$ (resp. $(-, -)_{\overline{A}}$) is the cup-product. Since $(\overline{E}_{1}\times\overline{E}_{2})[2]\subset(E_{1}\times E_{2})(\mathbb{F}_{q})$, this means $P_{1}, \cdots, P_{16}\in A(\mathbb{F}_{q})$. Hence the exceptional divisors $D_{1}, \cdots, D_{16}$ are defined over $\mathbb{F}_{q}$, i.e., $[D_{i}]\in{\rm NS}(X)\otimes \mathbb{Q}$. By Lemma \ref{abelalg} and the isometry \eqref{kummercoh}, we have ${\rm NS}(\overline{X})\otimes \mathbb{Q}={\rm NS}(\overline{A})\otimes \mathbb{Q}\oplus \bigoplus^{16}_{i=1}\mathbb{Q}[D_{i}]={\rm NS}(A)\otimes \mathbb{Q}\oplus \bigoplus^{16}_{i=1}\mathbb{Q}[D_{i}]={\rm NS}(X)\otimes\mathbb{Q}$. 
\end{proof}
\begin{prop}\label{formula} Assume that $\overline{E}_{1}$ and $\overline{E}_{2}$ are not isogenous. Suppose $\overline{E}_{i}[2]:=\{P\in E_{i}(\overline{\mathbb{F}}_{q})|~2P=O\}\subset E_{i}(\mathbb{F}_{q})$ for $i=1, 2$. Then, for every  odd prime $p$ and $n\geq0$, we have  
    \begin{equation}
        \#{\rm Br}(X_{n})[p^{\infty}]=p^{2v_{p}(\# E_{1}(\mathbb{F}_{q^{p^{n}}})-\# E_{2}(\mathbb{F}_{q^{p^{n}}}))}.
    \end{equation} 
\end{prop}
\begin{proof}
    We first show that the discriminant of ${\rm NS}(\overline{X})$ is a power of $2$. Consider the sublattice
    \begin{equation}
        L:=(\mathbb{Z}[\overline{E}_{1}\times \{O_{2}\}]\oplus \mathbb{Z}[\{O_{1}\}\times \overline{E}_{2}], 2(-, -)_{\overline{A}})\oplus \bigoplus^{16}_{i=1}\mathbb{Z}[D_{i}] \subset ({\rm NS}(\overline{X}), (-, -)_{\overline{X}}).
    \end{equation}
   The self-intersection $([\overline{E}_{1}\times \{O_{2}\}])^{2}_{\overline{A}}$ is equal to $0$. Indeed, choose $P_{2}\in E_{2}(\overline{{\mathbb{F}}}_{q})$ such that $P_{2}\neq O_{2}$. Since $[\overline{E}_{1}\times \{O_{2}\}]$ is algebraically equivalent to $[E_{1}\times \{P_{2}\}]$, we have $([\overline{E}_{1}\times \{O_{2}\}])^{2}_{\overline{A}}=([\overline{E}_{1}\times\{O_{2}\}])_{\overline{A}}.([\overline{E}_{1}\times \{O_{2}\}])_{\overline{A}}=0$, where $(-).(-)$ denotes the intersection number. Likewise, we have $([\{O_{1}\}\times \overline{E}_{2}])^{2}_{\overline{A}}=0$. On the other hand, since $O_{i}$ is $\overline{\mathbb{F}}_{q}$-rational, that is, the residue field $\kappa(O_{i})$ is $\overline{\mathbb{F}}_{q}$. By \cite[Lemma 33.16.7]{stk}, there is a canonical isomorphism $T_{\overline{A}/\overline{\mathbb{F}}_{q}, (O_{1}, O_{2})}\cong T_{\overline{E}_{1}/\overline{\mathbb{F}}_{q}, O_{1}}\oplus T_{\overline{E}_{2}/\overline{\mathbb{F}}_{q}, O_{2}}$. Hence the subvarieties $\overline{E}_{1}\times\{O_{2}\}$ and $\{O_{1}\}\times\overline{E}_{2}$ intersect transversally. Thus, by \cite[V-Theorem 1.1]{har}, we obtain $([\overline{E}_{1}\times\{O_{2}\}])_{\overline{A}}.([\{O_{1}\}\times\overline{E}_{2}])_{\overline{A}}=1$. Hence we have
   \begin{equation}
       (\mathbb{Z}[\overline{E}_{1}\times\{O_{2}\}]\oplus \mathbb{Z}[\{O_{1}\}\times \overline{E}_{2}], 2(-, -)_{\overline{A}})\cong U(2),
   \end{equation}
   where $U=\begin{pmatrix}0&1\\1&0\end{pmatrix}$ is the hyperplane. Since $|{\rm disc}(L)|=2^{18}$, we have ${\rm disc}\,{\rm NS}(\overline{X})\mid2^{18}$. This shows that ${\rm disc}\,{\rm NS}(\overline{X})$ is a power of $2$.

   By the Artin-Tate conjecture  (Theorem \ref{artintate}) and Proposition \ref{nskummer}, we have
   \begin{equation}\label{brkummer}
       \#{\rm Br}
(X)[p^{\infty}]=p^{v_{p}({\rm det}(1-F:T_{p}{\rm Br}(\overline{X})^{\ast})}.
\end{equation}
Let $\alpha_{i}, \beta_{i}$ be the eigenvalues of the geometric Frobenius $F^{-1}$ on $H^{1}_{\text{\'et}}(\overline{E}_{i}, \mathbb{Q}_{p})$. The Weil conjecture for elliptic curves (for example, see \cite[V-Theorem 2.3.1]{sil}) says 
\begin{equation}
    \alpha_{i}+\beta_{i}=q+1-\# E_{i}(\mathbb{F}_{q}), ~~~\alpha_{i}\beta_{i}=q.
\end{equation}
By the K\"unneth formula, we have
\begin{equation}\label{kunneth}
    H^{2}_{\text{\'et}}(\overline{A}, \mathbb{Q}_{p}(1))\cong H^{2}_{\text{\'et}}(\overline{E}_{1}, \mathbb{Q}_{p}(1))\oplus H^{1}_{\text{\'et}}(\overline{E}_{1}, \mathbb{Q}_{p})\otimes H^{1}_{\text{\'et}}(\overline{E}_{2}, \mathbb{Q}_{p})(1)\oplus H^{2}_{\text{\'et}}(\overline{E}_{2}, \mathbb{Q}_{p}(1)).
\end{equation}
Since the cohomology class $[O_{1}]\in H^{2}_{\text{\'et}}(\overline{E}_{1}, \mathbb{Q}_{p}(1))$ (resp. $[O_{2}]\in H^{2}_{\text{\'et}}(\overline{E}_{2}, \mathbb{Q}_{p}(1))$) corresponds to the class $[\{O_{1}\}\times\overline{E}_{2}]\in H^{2}_{\text{\'et}}(\overline{A}, \mathbb{Q}_{p}(1))$ (resp. $[\overline{E}_{1}\times\{O_{2}\}]\in H^{2}_{\text{\'et}}(\overline{A}, \mathbb{Q}_{p}(1))$) via projections, the decomposition \eqref{kunneth} induces an isomorphism
\begin{equation}
    T(\overline{A})_{p}\cong H^{1}_{\text{\'et}}(\overline{E}_{1}, \mathbb{Q}_{p})\otimes H^{1}_{\text{\'et}}(\overline{E}_{2}, \mathbb{Q}_{p})(1),
\end{equation}
where $T(\overline{A})_{p}$ is the orthogonal complement of ${\rm NS}(\overline{A})\otimes\mathbb{Q}_{p}$ with respect to the cup-product of $H^{2}_{\text{\'et}}(\overline{A}, \mathbb{Q}_{p}(1))$. On the other hand, \eqref{kummercoh} says 
\begin{equation}
    T_{p}{\rm Br}(\overline{X})\cong T(\overline{A})_{p}\cong H^{1}_{\text{\'et}}(\overline{E}_{1}, \mathbb{Q}_{p})\otimes H^{1}_{\text{\'et}}(\overline{E}_{2}, \mathbb{Q}_{p})(1),
\end{equation}
and thus we have 
\begin{align}
    {\rm det}(1-F: T_{p}{\rm Br}(\overline{X})^{\ast})&={\rm det}(1-F^{-1}: T_{p}{\rm Br}(\overline{X}))
    \\
    &=(1-\frac{\alpha_{1}\alpha_{2}}{q})(1-\frac{\alpha_{1}\beta_{2}}{q})(1-\frac{\beta_{1}\alpha_{2}}{q})(1-\frac{\beta_{1}\beta_{2}}{q})
    \\&=\frac{(\alpha_{1}+\beta_{1}-(\alpha_{2}+\beta_{2}))^{2}}{q}
    \\&=\frac{(\# E_{1}(\mathbb{F}_{q})-\# E_{2}(\mathbb{F}_{q}))^{2}}{q}.
\end{align}
The consequence follows from this calculation and \eqref{brkummer}.
\end{proof}
\begin{ex} Consider the case $q=29$. Let $E_{i}/\mathbb{F}_{29}$ be elliptic curves given in the following table:
\begin{center}
\begin{tblr}{|l|r|r|r|r|} \hline
    \SetCell[c=1]{c}equation  & \SetCell[c=1]{c}  L-function & \SetCell[c=1]{c}${\rm End}^{0}(E_{i})$ &\SetCell[c=1]{c}$\#E_{i}(\mathbb{F}_{29})$&\SetCell[c=1]{c}LMFDB label \cite{lmfdb}\\ \hline
  $ E_{1}: y^{2}=x^{3}-x$ & $1+10t+29t^{2}$ &$\mathbb{Q}(\sqrt{-1})$ & 40 & 1.29.k \\
   $E_{2}: y^{2}=x^{3}+x^{2}-2x$ &$1-2t+29t^{2}$  &$\mathbb{Q}(\sqrt{-7})$ &28& 1.29.ac \\ \hline
 \end{tblr}
\end{center}
Since the endomorphism algebras are different  from each other, $E_{1}$ and $E_{2}$ are not isogenous over $\overline{\mathbb{F}}_{29}$. Since $x^{3}-x=x(x-1)(x+1)$, we have $\overline{E}_{1}[2]=\{O_{1}, (0, 0), (1, 0), (-1, 0)\}$. For the same reason, we have $\overline{E}_{2}[2]=\{O_{2}, (0, 0), (1, 0), (-2, 0)\}$. Hence, by Proposition \ref{formula},  we obtain $\#{\rm Br}(X_{n})[p^{\infty}]=p^{2v_{p}(\# E_{1}(\mathbb{F}_{29^{p^{n}}})-\# E_{2}(\mathbb{F}_{29^{p^{n}}}))}$ for every $n\geq0$. We now specialize to the case $p=3$. Let $\alpha_{i}, \beta_{i}$ be the eigenvalues of the L-function of $E_{i}/\mathbb{F}_{29}$. Then the eigenvalues of the $L$-function for $E_{i}\times_{\mathbb{F}_{29}}{\rm Spec}(\mathbb{F}_{29^{3^{n}}})$ are given by $\alpha_{i}^{3^n}$ and $\beta_{i}^{3^n}$. Since $\alpha_{i}\beta_{i}=29$, we have 
\begin{equation}
    \alpha^{3}_{i}+\beta^{3}_{i}=(\alpha_{i}+\beta_{i})^{3}-3\cdot 29(\alpha_{i}+\beta_{i}).
\end{equation}
This implies 
\begin{equation}
    \alpha^{3^{n}}_{i}+\beta^{3^{n}}_{i}=(\alpha^{3^{n-1}}_{i}+\beta^{3^{n-1}}_{i})^{3}-3\cdot 29^{3^{n-1}}(\alpha^{3^{n-1}}_{i}+\beta^{3^{n-1}}_{i}).
\end{equation}
Let $d_{n}:=\alpha^{3^{n}}_{1}+\beta^{3^{n}}_{1}-(\alpha^{3^{n}}_{2}+\beta^{3^{n}}_{2})$. Then we have
\begin{align}
    d_{n}&=(\alpha^{3^{n-1}}_{1}+\beta^{3^{n-1}}_{1})^{3}-3\cdot 29^{3^{n-1}}(\alpha^{3^{n-1}}_{1}+\beta^{3^{n-1}}_{1})\\
    &\hspace{15em}-((\alpha^{3^{n-1}}_{2}+\beta^{3^{n-1}}_{2})^{3}-3\cdot 29^{3^{n-1}}(\alpha^{3^{n-1}}_{2}+\beta^{3^{n-1}}_{2}))
    \\&=(\alpha^{3^{n-1}}_{1}+\beta^{3^{n-1}}_{1})^{3}-(\alpha^{3^{n-1}}_{2}+\beta^{3^{n-1}}_{2})^{3}-3\cdot 29^{3^{n-1}}(\alpha^{3^{n-1}}_{1}+\beta^{3^{n-1}}_{1}-(\alpha^{3^{n-1}}_{2}+\beta^{3^{n-1}}_{2}))
    \\&=d_{n-1}((\alpha^{3^{n-1}}_{1}+\beta^{3^{n-1}}_{1})^{2}+(\alpha^{3^{n-1}}_{1}+\beta^{3^{n-1}}_{1})(\alpha^{3^{n-1}}_{2}+\beta^{3^{n-1}}_{2})+(\alpha^{3^{n-1}}_{2}+\beta^{3^{n-1}}_{2})^{2}-3\cdot29^{3^{n-1}}).
\end{align}
By induction, using $v_{3}(d_{0})=v_{3}(-40+28)=v_{3}(-12)=1$, we  see that $v_{3}(d_{n})=n+1$. Hence,
\begin{equation}
    \#{\rm Br}(X_{n})[3^{\infty}]=3^{2n+2}.
\end{equation}
This means that $\mu=0, \lambda=2, \nu=2.$
    
\end{ex}

\section{Iwasawa Main Conjecture}
In this section, we prove the Iwasawa main conjecture (Theorem \ref{Iwasawa main conj k3}).

For every $n\geq 0$, recall that $\Gamma_{n}$ is the Galois group ${\rm Gal}(\mathbb{F}_{q^{p^{n}}}/\mathbb{F}_{q})$. Let $\gamma_{n}\in\Gamma_{n}$ be the restriction of the arithmetic Frobenius $F$ to $\mathbb{F}_{q^{p^{n}}}$. 

Let $\Lambda_{n}:=\mathbb{Z}_{p}[\Gamma_{n}]\cong \Lambda/(\gamma^{p^{n}}-1)\Lambda$. In this section, we consider the $\Lambda_{n}$-module: 
\begin{equation}\label{lambdan}
(T_{p}{\rm Br}(\overline{X})^{\ast}\otimes_{\mathbb{Z}_{p}}\Lambda_{n})/(1-F\gamma^{-1}_{n})(T_{p}{\rm Br}(\overline{X})^{\ast}\otimes_{\mathbb{Z}_{p}}\Lambda_{n}).
\end{equation}
The modules in the form \eqref{lambdan} form an inverse system via the restriction $\Lambda_{n}\to \Lambda_{m}$ for $n\geq m$. 
\begin{prop}\label{fgamma}There exists a canonical isomorphism of $\Lambda_{n}$-modules:
\begin{equation}
T_{p}{\rm Br}(\overline{X})^{\ast}/(1-F^{p^{n}})T_{p}{\rm Br}(\overline{X})^{\ast}\cong (T_{p}{\rm Br}(\overline{X})^{\ast}\otimes_{\mathbb{Z}_{p}}\Lambda_{n})/(1-F\gamma^{-1}_{n})(T_{p}{\rm Br}(\overline{X})^{\ast}\otimes_{\mathbb{Z}_{p}}\Lambda_{n}).
\end{equation}
\end{prop}
\begin{proof} In the module $(T_{p}{\rm Br}(\overline{X})^{\ast}\otimes_{\mathbb{Z}_{p}}\Lambda_{n})/(1-F\gamma^{-1}_{n})(T_{p}{\rm Br}(\overline{X})^{\ast}\otimes_{\mathbb{Z}_{p}}\Lambda_{n})$, we have 
\begin{equation}\label{rel}
F\alpha\otimes f=\alpha \otimes \gamma_{n}f
\end{equation}
for all $\alpha\in T_{p}{\rm Br}(\overline{X})^{\ast}, f\in \Lambda_{n}$. In particular, we have $F^{p^{n}}\alpha\otimes f=\alpha\otimes \gamma^{p^{n}}_{n}f=\alpha\otimes f$. Hence the natural $G$-homomorphism $T_{p}{\rm Br}(\overline{X})^{\ast}\to (T_{p}{\rm Br}(\overline{X})^{\ast}\otimes_{\mathbb{Z}_{p}}\Lambda_{n})/(1-F\gamma^{-1}_{n})(T_{p}{\rm Br}(\overline{X})^{\ast}\otimes_{\mathbb{Z}_{p}}\Lambda_{n}), \alpha\mapsto \alpha\otimes 1$ induces a $\Lambda_{n}$-homomorphism:
\begin{equation}\label{rel1}
T_{p}{\rm Br}(\overline{X})^{\ast}/(1-F^{p^{n}})T_{p}{\rm Br}(\overline{X})^{\ast}\to  (T_{p}{\rm Br}(\overline{X})^{\ast}\otimes_{\mathbb{Z}_{p}}\Lambda_{n})/(1-F\gamma^{-1}_{n})(T_{p}{\rm Br}(\overline{X})^{\ast}\otimes_{\mathbb{Z}_{p}}\Lambda_{n}).
\end{equation}
Conversely, let $\overline{\alpha}$ be the class of $\alpha\in T_{p}{\rm Br}(\overline{X})^{\ast}$ in the quotient $T_{p}{\rm Br}(\overline{X})^{\ast}/(1-F^{p^{n}})T_{p}{\rm Br}(\overline{X})^{\ast}$. In the quotient $(T_{p}{\rm Br}(\overline{X})^{\ast}\otimes_{\mathbb{Z}_{p}}\Lambda_{n})/(1-F\gamma^{-1}_{n})(T_{p}{\rm Br}(\overline{X})^{\ast}\otimes_{\mathbb{Z}_{p}}\Lambda_{n})$, we obtain $\overline{F\alpha}=\gamma_{n}\overline{\alpha}$. Hence the natural $\Lambda_{n}$-homomorphism  $T_{p}{\rm Br}(\overline{X})^{\ast}\otimes_{\mathbb{Z}_{p}}\Lambda_{n} \to T_{p}{\rm Br}(\overline{X})^{\ast}/(1-F^{p^{n}})T_{p}{\rm Br}(\overline{X})^{\ast}$ via $\alpha\otimes \gamma_{n}\mapsto \gamma_{n}\overline{\alpha}$ induces a $\Lambda_{n}$-homomorphism
\begin{equation}\label{rel2}
(T_{p}{\rm Br}(\overline{X})^{\ast}\otimes_{\mathbb{Z}_{p}}\Lambda_{n})/(1-F\gamma^{-1}_{n})(T_{p}{\rm Br}(\overline{X})^{\ast}\otimes_{\mathbb{Z}_{p}}\Lambda_{n})\to T_{p}{\rm Br}(\overline{X})^{\ast}/(1-F^{p^{n}})T_{p}{\rm Br}(\overline{X})^{\ast}.
\end{equation}
The morphisms \eqref{rel1} and \eqref{rel2} give the inverse maps of each other. 
\end{proof}
\begin{prop}\label{fgammainj1} The $\Lambda_{n}$-module homomorphism
\begin{equation}\label{fgammainj}
1-F\gamma_{n}^{-1}: T_{p}{\rm Br}(\overline{X})^{\ast}\otimes_{\mathbb{Z}_{p}}\Lambda_{n}\to T_{p}{\rm Br}(\overline{X})^{\ast}\otimes_{\mathbb{Z}_{p}}\Lambda_{n}
\end{equation}
is injective.
\end{prop}
\begin{proof} By Proposition \ref{dualtate}, Theorem \ref{ls}, and Proposition \ref{fgamma}, the module $(T_{p}{\rm Br}(\overline{X})^{\ast}\otimes_{\mathbb{Z}_{p}}\Lambda_{n})/(1-F\gamma^{-1}_{n})(T_{p}{\rm Br}(\overline{X})^{\ast}\otimes_{\mathbb{Z}_{p}}\Lambda_{n})$ is finite. In particular, the extension of scalars to $\mathbb{Q}_{p}$
\begin{equation}
(1-F\gamma_{n}^{-1})\otimes{\rm id}: (T_{p}{\rm Br}(\overline{X})^{\ast}\otimes_{\mathbb{Z}_{p}}\Lambda_{n})\otimes_{\mathbb{Z}_{p}}\mathbb{Q}_{p}\to(T_{p}{\rm Br}(\overline{X})^{\ast}\otimes_{\mathbb{Z}_{p}}\Lambda_{n})\otimes_{\mathbb{Z}_{p}}\mathbb{Q}_{p}
\end{equation}
is an isomorphism, i.e., an injection. Since $T_{p}{\rm Br}(\overline{X})^{\ast}\otimes_{\mathbb{Z}_{p}}\Lambda_{n}$ is a free $\mathbb{Z}_{p}$-module, we  see that the morphism \eqref{fgammainj} is injective. 
\end{proof}
\begin{prop}\label{fgamma2} There is a canonical isomorphism of $\Lambda$-modules:
\begin{equation}
\varprojlim_{n}\frac{(T_{p}{\rm Br}(\overline{X})^{\ast}\otimes_{\mathbb{Z}_{p}}\Lambda_{n})}{(1-F\gamma^{-1}_{n})(T_{p}{\rm Br}(\overline{X})^{\ast}\otimes_{\mathbb{Z}_{p}}\Lambda_{n})}\cong \frac{(T_{p}{\rm Br}(\overline{X})^{\ast}\otimes_{\mathbb{Z}_{p}}\Lambda)}{(1-F\gamma^{-1}) (T_{p}{\rm Br}(\overline{X})^{\ast}\otimes_{\mathbb{Z}_{p}}\Lambda)}.
\end{equation}
\end{prop}
\begin{proof}
Clearly, for $n\geq m$, the restriction $(1-F\gamma^{-1}_{n})(T_{p}{\rm Br}(\overline{X})\otimes_{\mathbb{Z}_{p}}\Lambda_{n})\to (1-F\gamma^{-1}_{m})(T_{p}{\rm Br}(\overline{X})\otimes_{\mathbb{Z}_{p}}\Lambda_{m})$ is surjective. Hence the inverse system $((1-F\gamma^{-1}_{n})(T_{p}{\rm Br}(\overline{X})\otimes_{\mathbb{Z}_{p}}\Lambda_{n})_{n\geq 0}$ is  Mittag-Leffler (\cite[Section 10.86, Example 10.86.2]{stk}). In particular, we have a canonical isomorphism of $\Lambda$-modules:
\begin{equation}
\varprojlim_{n}\frac{(T_{p}{\rm Br}(\overline{X})^{\ast}\otimes_{\mathbb{Z}_{p}}\Lambda_{n})}{(1-F\gamma^{-1}_{n})(T_{p}{\rm Br}(\overline{X})^{\ast}\otimes_{\mathbb{Z}_{p}}\Lambda_{n})}\cong \frac{\varprojlim_{n}(T_{p}{\rm Br}(\overline{X})^{\ast}\otimes_{\mathbb{Z}_{p}}\Lambda_{n})}{\varprojlim_{n}(1-F\gamma^{-1}_{n})(T_{p}{\rm Br}(\overline{X})^{\ast}\otimes_{\mathbb{Z}_{p}}\Lambda_{n})}.
\end{equation}
The restriction $(1-F\gamma^{-1})(T_{p}{\rm Br}(\overline{X})^{\ast}\otimes_{\mathbb{Z}_{p}}\Lambda)\to (1-F\gamma_{n}^{-1})(T_{p}{\rm Br}(\overline{X})^{\ast}\otimes_{\mathbb{Z}_{p}}\Lambda_{n})$ induces a canonical morphism of $\Lambda$-modules
\begin{equation}\label{fgammafun}
(1-F\gamma^{-1})(T_{p}{\rm Br}(\overline{X})^{\ast}\otimes_{\mathbb{Z}_{p}}\Lambda)\to \varprojlim_{n}(1-F\gamma_{n}^{-1})(T_{p}{\rm Br}(\overline{X})^{\ast}\otimes_{\mathbb{Z}_{p}}\Lambda_{n}),
\end{equation}
and we have the following commutative diagram:
\[
\begin{tikzcd}
T_{p}{\rm Br}(\overline{X})^{\ast}\otimes_{\mathbb{Z}_{p}}\Lambda\ar[r, "\sim"] & \varprojlim_{n}(T_{p}{\rm Br}(\overline{X})^{\ast}\otimes_{\mathbb{Z}_{p}}\Lambda_{n})\\
    (1-F\gamma^{-1})(T_{p}{\rm Br}(\overline{X})^{\ast}\otimes_{\mathbb{Z}_{p}}\Lambda)\ar{r}\ar[u, hook] & \varprojlim_{n}(1-F\gamma_{n}^{-1})(T_{p}{\rm Br}(\overline{X})^{\ast}\otimes_{\mathbb{Z}_{p}}\Lambda_{n}), \ar[u, hook]
\end{tikzcd}
\]
where the top horizontal morphism is an isomorphism because $T_{p}{\rm Br}(\overline{X})$ is a free $\mathbb{Z}_{p}$-module of finite rank.
In particular, the morphism \eqref{fgammafun} is injective. To prove the surjectivity, let $((1-F\gamma_{n}^{-1})(\alpha_{n}))_{n\geq 0}\in \varprojlim_{n}(1-F\gamma_{n}^{-1})(T_{p}{\rm Br}(\overline{X})^{\ast}\otimes_{\mathbb{Z}_{p}}\Lambda_{n})$, where $\alpha_{n}\in T_{p}{\rm Br}(\overline{X})^{\ast}\otimes_{\mathbb{Z}_{p}}\Lambda_{n}$. For $n\geq m$, we have $(1-F\gamma^{-1}_{m})(\alpha_{n}|_{T_{p}{\rm Br}(\overline{X})^{\ast}\otimes_{\mathbb{Z}_{p}}\Lambda_{m}})=(1-F\gamma^{-1}_{m})(\alpha_{m})$.  Hence, by Proposition \ref{fgammainj1}, we have $\alpha_{n}|_{T_{p}{\rm Br}(\overline{X})^{\ast}\otimes_{\mathbb{Z}_{p}}\Lambda_{m}}=\alpha_{m}$. Thus there exists a unique $\alpha\in T_{p}{\rm Br}(\overline{X})^{\ast}\otimes_{\mathbb{Z}_{p}}\Lambda$ such that $\alpha|_{T_{p}{\rm Br}(\overline{X})\otimes_{\mathbb{Z}_{p}}\Lambda_{n}}=\alpha_{n}$ for every $n\geq 0$. Then 
\begin{equation}
    (1-F\gamma^{-1})(\alpha)|_{(1-F\gamma_{n}^{-1})(T_{p}{\rm Br}(\overline{X})\otimes_{\mathbb{Z}_{p}}\Lambda_{n})}=(1-F\gamma^{-1}_{n})(\alpha|_{(1-F\gamma_{n}^{-1})(T_{p}{\rm Br}(\overline{X})\otimes_{\mathbb{Z}_{p}}\Lambda_{n})})=(1-F\gamma_{n}^{-1})\alpha_{n},
\end{equation}
that is, the morphism \eqref{fgammafun} is an isomorphism. Hence we obtain
\begin{equation}
    \frac{\varprojlim_{n}(T_{p}{\rm Br}(\overline{X})^{\ast}\otimes_{\mathbb{Z}_{p}}\Lambda_{n})}{\varprojlim_{n}(1-F\gamma^{-1}_{n})(T_{p}{\rm Br}(\overline{X})^{\ast}\otimes_{\mathbb{Z}_{p}}\Lambda_{n})}\cong \frac{T_{p}{\rm Br}(\overline{X})^{\ast}\otimes_{\mathbb{Z}_{p}}\Lambda}{(1-F\gamma^{-1})(T_{p}{\rm Br}(\overline{X})^{\ast}\otimes_{\mathbb{Z}_{p}}\Lambda)}.
\end{equation}
\end{proof}
\begin{prop}\label{fitarith} There exists a canonical $\Lambda$-isomorphism:
\begin{equation}
{\rm Br}(\overline{X})^{G^{\infty}}[p^{\infty}]^{\lor}\cong (T_{p}{\rm Br}(\overline{X})^{\ast}\otimes_{\mathbb{Z}_{p}}\Lambda)/(1-F\gamma^{-1}) (T_{p}{\rm Br}(\overline{X})^{\ast}\otimes_{\mathbb{Z}_{p}}\Lambda). 
\end{equation}
\end{prop}
\begin{proof} 
By Proposition \ref{dualtate}, we have 
\begin{equation}
    {\rm Br}(\overline{X})^{G^{\infty}}[p^{\infty}]^{\lor}\cong\varprojlim_{n}{\rm Br}(\overline{X})^{G^{n}}[p^{\infty}]^{\lor}\cong\varprojlim_{n}\frac{T_{p}{\rm Br}(\overline{X})^{\ast}}{(1-F^{p^{n}})T_{p}{\rm Br}(\overline{X})^{\ast}}.
\end{equation}
On the other hand, by Proposition \ref{fgamma} and \ref{fgamma2}, we have
\begin{equation}
    \varprojlim_{n}\frac{T_{p}{\rm Br}(\overline{X})^{\ast}}{(1-F^{p^{n}})T_{p}{\rm Br}(\overline{X})^{\ast}}\cong \varprojlim_{n}\frac{(T_{p}{\rm Br}(\overline{X})^{\ast}\otimes_{\mathbb{Z}_{p}}\Lambda_{n})}{(1-F\gamma^{-1}_{n})(T_{p}{\rm Br}(\overline{X})^{\ast}\otimes_{\mathbb{Z}_{p}}\Lambda_{n})}\cong\frac{T_{p}{\rm Br}(\overline{X})^{\ast}\otimes_{\mathbb{Z}_{p}}\Lambda}{(1-F\gamma^{-1})(T_{p}{\rm Br}(\overline{X})^{\ast}\otimes_{\mathbb{Z}_{p}}\Lambda)}.
\end{equation}
\end{proof}
\begin{thm}[Iwasawa main conjecture for \text{${\rm Br}(X_{\infty})[p^{\infty}]^{\lor}$}]\label{imck3}  Let $X/\mathbb{F}_{q}$ be a K3 surface, and let $X_{\infty}:=X\times_{\mathbb{F}_{q}}{\rm Spec}(\mathbb{F}_{q^{p^{\infty}}})$. Fix the isomorphism $\Lambda\cong \mathbb{Z}_{p}[[T]]$ induced by the arithmetic Frobenius $F$. Then we have an equality of ideals in $\Lambda$
\begin{equation}
{\rm Char}_{\Lambda}({\rm Br}(X_{\infty})[p^{\infty}]^{\lor})=(L_{{\rm tr}}(X, 1+T)).
\end{equation} 
\end{thm}
\begin{proof} By Proposition \ref{fitarith} and \cite[Proposition 4.3]{tu}, we have
\begin{equation}
    {\rm Char}_{\Lambda}({\rm Br}(\overline{X})^{G^{\infty}}[p^{\infty}]^{\lor})=({\rm det}(1-F\gamma^{-1}: T_{p}{\rm Br}(\overline{X})^{\ast}\otimes_{\mathbb{Z}_{p}}\Lambda)).
\end{equation}
Under the isomorphism $\Lambda\cong \mathbb{Z}_{p}[[T]]$ via $\gamma\mapsto 1+T$, we have
\begin{align}
    ({\rm det}(1-F\gamma^{-1}: T_{p}{\rm Br}(\overline{X})^{\ast}\otimes_{\mathbb{Z}_{p}}\Lambda))&=({\rm det}(1-F(1+T)^{-1}: T_{p}{\rm Br}(\overline{X})^{\ast}\otimes_{\mathbb{Z}_{p}}\mathbb{Z}_{p}[[T]]))
    \\&=({\rm det}(1-F^{-1}(1+T):T_{p}{\rm Br}(\overline{X})\otimes_{\mathbb{Z}_{p}}\mathbb{Z}_{p}[[T]]))
    \\&=(L_{{\rm tr}}(X, 1+T)).
\end{align}
By Theorem \ref{geoarith},  the finitely generated torsion $\Lambda$-module ${\rm Br}(X_{\infty})[p^{\infty}]^{\lor}$ is pseudo-isomorphic to the $\Lambda$-module ${\rm Br}(\overline{X})^{G^{\infty}}[p^{\infty}]^{\lor}$. Hence we obtain
\begin{equation}
    {\rm Char}_{\Lambda}({\rm Br}(X_{\infty})[p^{\infty}]^{\lor})={\rm Char}_{\Lambda}({\rm Br}(\overline{X})^{G^{\infty}}[p^{\infty}]^{\lor})=(L_{{\rm tr}}(X, 1+T)).
\end{equation}
\end{proof}
The Artin--Tate conjecture is expected to hold for all smooth geometrically connected projective surfaces over finite fields. A natural question is whether the Iwasawa theory developed in this paper can be generalized to this broader class of surfaces. We believe that this direction is important for the further development of the subject.


\begin{thebibliography}{45}
\bibitem{cassels} J. W. S. Cassels and A. Fr\"ohlich., ed, \emph{Algebraic Number Theory}, second edition, London Mathematical Society, 2010. 
\bibitem{cs} J. -L. Colliot-Th\`el\'ene and A. N. Skorobogatov, \emph{The Brauer-Grothendieck Group}, Ergebnisse der Mathematik und ihrer Grenzgebiete, 3. Folge, Band 71, Springer, 2021.
\bibitem{fer} B. Ferrero and L. C. Washington, \emph{The Iwasawa Invariant $\mu_{p}$ Vanishes for Abelian Number Fields},
Ann. Math. 109, No. 2, 377--395, 1979. 
\bibitem{gre} R. Greenberg, \emph{Iwasawa Theory for elliptic curves}, Lecture Notes in Math. 1716, Springer, 1999.
\bibitem{sga4}  A. Grothendieck, et al., SGA 4 (with M. Artin and J. L. Verdier) \emph{Th\'eorie des Topos et Cohomologie Etale des Sch\'emas}, Lecture Notes in Math. 270, Springer-Verlag, Heidelberg 1972.  
\bibitem{dix}  A. Grothendieck, Le groupe de Brauer. II. Th\'eorie cohomologique. in \emph{Dix Expos\'es sur la Cohomologie des Sch\'emas}, 88--188, North-Holland, Amsterdam, 1968.
\bibitem{har} R. Hartshorne, \emph{Algebraic Geometry}, Graduate Texts in Mathematics, vol. 52, Springer-Verlag, New York, 1977.
\bibitem{Huy} D. Huybrechts, \emph{Lectures on K3 surfaces}, volume 158 of Cambridge Studies in Advanced Mathematics. Cambridge University Press, Cambridge, 2016.
\bibitem{iwasawa} K. Iwasawa, \emph{On $\Gamma$-extensions of algebraic number fields}, Bull. Amer. Math. Soc. 65, 183--226, 1959.
\bibitem{iwasawa2} K. Iwasawa,  \emph{On the $\mu$-invariants of $\mathbb{Z}_{l}$-extensions}, Number theory, algebraic geometry and commutative algebra, in honor of Yasuo Akizuki, 1--11, 1973.
\bibitem{kato} K. Kato, \emph{$p$-adic Hodge theory and values of zeta functions of modular forms}, Ast\'erisque, 295, 117--290, 2004.
\bibitem{kmp} W. Kim and K. Madapusi Pera, \emph{2-adic integral canonical models}, Forum Math. Sigma 4 (2016), e28, 34.
\bibitem{suzuki} King-Fai Lai, Ignazio Longhi, Takashi Suzuki, Ki-Seng Tan and Fabien Trihan, On the $\mu$-invariants of abelian varieties
over function fields of positive characteristic, Algebra \& Number Theory, 15, no.14, 2021.  
\bibitem{lmfdb} The LMFDB Collaboration (2026). The L-functions and modular forms database.
\bibitem{mp15} K. Madapusi Pera, \emph{The Tate conjecture for K3 surfaces in odd characteristic}, Invent. Math. 205 (2015), no.2, 625--668.
\bibitem{matsumura} H. Matsumura, \emph{Commutative Algebra}, second edition, The Benjamin/Cummings Publishing Company, Inc., 1980.
\bibitem{mazur}B. Mazur, \emph{Rational Points of Abelian Varieties
with Values in Towers of Number Fields}, Invent. math. 18, 183-266 (1972). 
\bibitem{ms} B. Mazur and P. Swinnerton-Dyer, \emph{Arithmetic of Weil Curves}, lnvent. math. 25, 1--61, 1974. 
\bibitem{mw} B. Mazur and A. Wiles, \emph{Class fields of Abelian extensions of $\mathbb{Q}$}, Invent. math. 76, 179--330, 1984. 
\bibitem{milne0} J. S. Milne, Abelian Varieties in \emph{Arithmetic Geometry} ed by G. Cornell and J. H. Silverman, Springer-Verlag New York Inc. 1986, 103--150.
\bibitem{milne} J. S. Milne, \emph{\'Etale Cohomology}, Princeton Mathematical Series, 33, Princeton University Press, 1980.
\bibitem{nsw} J. Neukirch, A. Schmidt, and K. Wingberg, \emph{Cohomology of Number Fields},  Grundlehren der mathematischen Wissenschaften 323, Springer, 2000.
\bibitem{o1} T. Ochiai, \emph{Iwasawa Theory and its perspective I}, Iwanami Studies in Advanced Mathematics, Iwanami Shoten, 2014.
\bibitem{pont} L. S. Pontryagin, \emph{Topological Groups}, Gordon and Breach, New York, London, Paris 1966.
\bibitem{rubin} K. Rubin, \emph{The ``main conjectures'' of Iwasawa theory for
imaginary quadratic fields}, Invent. math. 103, 25--68 (1991).
\bibitem{serre} J. P. Serre, \emph{Corps Locaux}, Hermann Paris, 1962.
\bibitem{sil} J. H. Silverman, \emph{The Arithmetic of Elliptic Curves}, second ed., Graduate Texts in Mathematics, vol. 106, Springer-Verlag, New York, 2009. 
\bibitem{su} C. Skinner and E. Urban, \emph{The Iwasawa Main Conjecture for $GL_{2}$}, Invent. Math. 195, 1, 1--277, 2014. 
\bibitem{stk} The Stacks Project Authors, \emph{Stacks Project}, https://stacks.math.columbia.edu, 2018.
\bibitem{tate} J. Tate, \emph{On the conjectures of Birch and Swinnerton-Dyer and a geometric analog}, S\'eminaire Bourbaki, Expos\'e 306, 1964/1966. 415--440, Soci\'et\'e Math\'ematique de France, 1995. 
\bibitem{tu} S. Tateno and J. Ueki, \emph{The Iwasawa invariants of $\mathbb{Z}^{d}_{p}$-covers of links}, Journal of the London Mathematical Society, volume 111, issue 6, 2025. 
\bibitem{Washington} L. C. Washington, \emph{Introduction to Cyclotomic Fields}, second ed., Graduate Texts in Mathematics, vol. 83, Springer-Verlag, New York, 1997.
\bibitem{wiles1} A. Wiles, \emph{The Iwasawa conjecture for totally real fields}, Ann. Math. (2) 131, No. 3, 493--540, 1990.
\end{thebibliography}
\end{document}